\documentclass[preprint, 12pt, sort&compress]{elsarticle}

\usepackage{latexsym}
\usepackage{amsfonts}
\usepackage{graphicx,amssymb,natbib}
\usepackage{amsmath}
\usepackage{amssymb}
\usepackage[mathscr]{eucal}
\usepackage{enumerate}
\usepackage{amsthm}

\usepackage{bm}

\usepackage{color, soul}

\journal{Soft Computing}

\begin{document}

\newtheorem{tm}{Theorem}[section]
\newtheorem{pp}[tm]{Proposition}
\newtheorem{lm}[tm]{Lemma}
\newtheorem{df}[tm]{Definition}
\newtheorem{tl}[tm]{Corollary}
\newtheorem{re}[tm]{Remark}
\newtheorem{eap}[tm]{Example}

\newcommand{\pof}{\noindent {\bf Proof} }
\newcommand{\ep}{$\quad \Box$}

\newcommand{\al}{\alpha}
\newcommand{\be}{\beta}
\newcommand{\var}{\varepsilon}
\newcommand{\la}{\lambda}
\newcommand{\de}{\delta}
\newcommand{\str}{\stackrel}
\newcommand{\rmn}{\romannumeral}


\renewcommand{\proofname}{\bf Proof}

\allowdisplaybreaks

\begin{frontmatter}

\title{Characterizations of compactness of fuzzy set space with endograph metric
}
\author{Huan Huang}
 \ead{hhuangjy@126.com }
\address{Department of Mathematics, Jimei
University, Xiamen 361021, China}

\date{}

\begin{abstract}

In this paper, we present the characterizations of total boundedness, relative compactness and compactness in fuzzy set spaces equipped with
the endograph metric.
The conclusions in this paper
 significantly improve
the corresponding conclusions given in our previous paper
[H. Huang, Characterizations of endograph metric and $\Gamma$-convergence on fuzzy sets, Fuzzy Sets and Systems 350 (2018),  55-84].
The results in this paper are applicable to
fuzzy sets in a general metric space. The results in our previous paper are applicable to
fuzzy sets in the $m$-dimensional Euclidean space $\mathbb{R}^m$,
which is a specfic metric space.
Furthermore, based on the above results,
we
 give
 the characterizations of relative compactness,
total boundedness and compactness in a kind of common subspaces of general fuzzy sets according to the endograph metric.
As an application, we investigate some relationship between
the endograph metric and the $\Gamma$-convergence on fuzzy sets.

\end{abstract}


\end{frontmatter}

\date{}

\section{Introduction}

Fuzzy set is a fundamental tool to investigate fuzzy phenomenon
\cite{da, du, wu, garcia, wa, gong, wangun, shin}.
A fuzzy set can be identified with its endograph.
The endograph metric $H_{\rm end}$
on fuzzy sets
is the Hausdorff metric defined on their endographs.
It's
shown that
the endograph metric on fuzzy sets
has significant advantages
 \cite{kloeden2, kloeden, kupka, can}.

Compactness is one of the central concepts in topology and analysis and is useful in applications (see \cite{kelley, wa}).
The characterizations
of
compactness in various fuzzy set spaces endowed with different topologies
have attracted much attention \cite{greco,greco3,huang,huang9,huang719,roman,trutschnig}.

In \cite{huang},
we
have given the characterizations of total boundedness, relative compactness and compactness
of
fuzzy set spaces equipped with the endograph metric $H_{\rm end}$.

The results in \cite{huang} are applicable to fuzzy sets in the $m$-dimensional Euclidean space $\mathbb{R}^m$.
$\mathbb{R}^m$ is a specific metric space.
In theoretical research and practical applications,
fuzzy sets in a general metric space
are often used \cite{da, du, greco, greco3}.

In this paper, we present the characterizations of total boundedness, relative compactness and compactness
of the space
fuzzy sets in a general metric space equipped with the endograph metric $H_{\rm end}$.
We point out that the
 characterizations of total boundedness, relative compactness and compactness
given in \cite{huang}
are corollaries of
the corresponding characterizations given in this paper.

Furthermore,
we discuss
 the properties of the endograph metric $H_{\rm end}$,
and then use these properties
and the above characterizations for general fuzzy sets
to give
 the characterizations of relative compactness,
total boundedness and compactness in a kind of common subspaces of general fuzzy sets according to the endograph metric $H_{\rm end}$.

As an application of the characterizations of compactness
given in this paper, we discuss the relationship
between $H_{\rm end}$ metric and $\Gamma$-convergence
on fuzzy sets.

The remainder of this paper is organized as follows.
In Section 2, we recall and give some basic notions and fundamental results related to fuzzy sets
and
the endograph metric and the $\Gamma$-convergence on them.
In Section 3, we give representation theorems for various kinds of fuzzy sets which are useful in this paper.
In Section \ref{cmg},
 we give the characterizations of
relatively compact sets,
totally bounded sets, and compact sets in
space of fuzzy sets in a general metric space equipped with the endograph metric, respectively.
In Section 5, based on the characterizations of compactness
given in Section \ref{cmg},
we give the
characterizations of
relatively compact sets,
totally bounded sets, and compact sets in
 a kind of common subspaces of the fuzzy set space discussed in Section \ref{cmg}.
 In Section 6, as an application of the characterizations of compactness
given in Section \ref{cmg}, we investigate the relationship
between the endograph metric and the $\Gamma$-convergence
on fuzzy sets.
At last, we draw conclusions
in Section 7.

\section{Fuzzy sets and endograph metric and $\Gamma$-convergence on them}

In this section, we recall and give some basic notions and fundamental results related to fuzzy sets
and
the endograph metric and the $\Gamma$-convergence on them.
Readers
can refer to \cite{wu, da, du, rojas, huang17} for related contents.

Let $\mathbb{N}$ denote the set of natural numbers (that is, the set of all positive integers).
Let $\mathbb{R} $ denote the set of real numbers.
Let $\mathbb{R}^m$, $m>1$, denote the
set $\{\langle x_1, \ldots, x_m \rangle: x_i\in \mathbb{R}, \ i=1,\ldots,m \}$.
In the sequel, $\mathbb{R} $ is also written as $\mathbb{R}^1$.

Throughout this paper, we suppose that $X$ \emph{is a nonempty set and it is endowed with a metric} $d$.
For simplicity,
we also use $X$
to \emph{denote the metric space} $(X, d)$.

The metric $\overline{d}$ on $X \times [0,1]$ is defined
as follows: for $(x,\al), (y, \beta) \in X \times [0,1]$,
$$  \overline{d } ((x,\al), (y, \beta)) = d(x,y) + |\al-\beta| .$$

Throughout this paper, we suppose that \emph{the metric on} $X\times[0,1]$ \emph{is} $\overline{d}$.
For simplicity,
we also use $X \times [0,1]$
to denote the metric space $(X \times [0,1], \overline{d})$.

Let $m\in \mathbb{N}$.
For simplicity,
 $\mathbb{R}^m$ is also used to denote the
$m$-dimensional Euclidean space;
$d_m$ is used to denote the Euclidean metric on $\mathbb{R}^m$;
$\mathbb{R}^m \times [0,1]$ is also used to denote
the metric space $(\mathbb{R}^m \times [0,1], \overline{d_m})$.

A fuzzy set $u$ in $X$ can be seen as a function $u:X \to [0,1]$.
A
subset $S$ of $X$ can be seen as a fuzzy set in $X$. If there is no confusion,
 the fuzzy set corresponding to $S$ is often denoted by $\chi_{S}$; that is,
\[ \chi_{S} (x) = \left\{
                    \begin{array}{ll}
                      1, & x\in S, \\
                      0, & x\in X \setminus S.
                    \end{array}
                  \right.
\]
For simplicity,
for
$x\in X$, we will use $\widehat{x}$ to denote the fuzzy set  $\chi_{\{x\}}$ in $X$.
In this paper, if we want to emphasize a specific metric space $X$, we will write the fuzzy set corresponding to $S$ in $X$ as
$S_{F(X)}$, and the fuzzy set corresponding to $\{x\}$ in $X$ as $\widehat{x}_{F(X)}$.

The symbol $F(X)$ is used
to
denote the set of
all fuzzy sets in $X$.
For
$u\in F(X)$ and $\al\in [0,1]$, let $\{u>\al \} $ denote the set $\{x\in X: u(x)>\al \}$, and let $[u]_{\al}$ denote the \emph{$\al$-cut} of
$u$, i.e.
\[
[u]_{\al}=\begin{cases}
\{x\in X : u(x)\geq \al \}, & \ \al\in(0,1],
\\
{\rm supp}\, u=\overline{    \{ u > 0 \}    }, & \ \al=0,
\end{cases}
\]
where $\overline{S}$
denotes
the topological closure of $S$ in $(X,d)$.

If $X$ is replaced by a nonempty set $Y$ in the above two paragraphs, then the definitions and notations in the above two paragraphs still apply except for the notation $[u]_0$.

The symbol $K(X)$ and
 $C(X)$ are used
to
 denote the set of all nonempty compact subsets of $X$ and the set of all nonempty closed subsets of $X$, respectively.
$P(X)$ is used to denote the power set of $X$, which is the set of all subsets of $X$.

Let
$F_{USC}(X)$
denote
the set of all upper semi-continuous fuzzy sets $u:X \to [0,1]$,
i.e.,
$$F_{USC}(X) :=\{ u\in F(X) : [u]_\al \in  C(X) \cup \{\emptyset\}  \  \mbox{for all} \   \al \in [0,1]   \}.  $$

Define
\begin{gather*}
F_{USCB}(X):=\{ u\in  F_{USC}(X): [u]_0 \in K(X)\cup\{\emptyset\} \},
\\
F_{USCG}(X):=\{ u\in  F_{USC}(X): [u]_\al \in K(X)\cup\{\emptyset\} \ \mbox{for all} \   \al\in (0,1] \}.
 \end{gather*}
Clearly,
 $$F_{USCB}(X) \subseteq  F_{USCG}(X)   \subseteq  F_{USC}(X).$$

Define
\begin{gather*}
F_{CON}(X) := \{ u\in F(X): \mbox{for all } \al\in (0,1], \ [u]_\al \mbox{ is connected in } X   \},
\\
F_{USCCON} (X) :=  F_{USC}(X) \cap F_{CON}(X),\\
F_{USCGCON}(X):= F_{USCG}(X) \cap F_{CON}(X).
\end{gather*}

Let $u\in F_{CON}(X)$. Then $[u]_0 = \overline{\cup_{\al>0} [u]_\al}$ is connected in $X$.
The proof is as follows.

 If $u=\chi_{\emptyset}$, then $[u]_0 = \emptyset$ is connected in $X$.
 If $u \not= \chi_{\emptyset}$,
 then there is an $\al\in (0,1]$ such that $[u]_\al \not= \emptyset$.
Note that $[u]_\beta \supseteq [u]_\al$ when $\beta\in [0,\al]$.
Hence $\cup_{0<\beta<\al} [u]_\beta$ is connected, and thus
 $[u]_0 = \overline{\cup_{0<\beta<\al} [u]_\beta}$ is connected.

So
 $$
F_{CON}(X) = \{ u\in F(X): \mbox{for all } \al\in [0,1], \ [u]_\al \mbox{ is connected in } X   \}.
$$

Let
$F^1_{USC}(X)$
denote
the family of all normal fuzzy sets in $F_{USC}(X)$,
i.e.,
$$F^1_{USC}(X) :=\{ u\in F(X) : [u]_\al \in  C(X)  \  \mbox{for all} \   \al \in [0,1]   \}.  $$

We introduce some subclasses of $F^1_{USC}(X)$, which will be discussed in this paper.
Define
\begin{gather*}
F^1_{USCB}(X):= F^1_{USC}(X) \cap F_{USCB}(X) ,
\\
F^1_{USCG}(X):= F^1_{USC}(X) \cap F_{USCG}(X),
\\
  F^1_{USCCON} (X) := F^1_{USC}(X) \cap F_{CON}(X), \\
  F^1_{USCGCON} (X) := F^1_{USCG} (X) \cap F_{CON} (X).
 \end{gather*}
Clearly,
\begin{gather*}
F^1_{USCB}(X) \subseteq  F^1_{USCG}(X)   \subseteq  F^1_{USC}(X),\\
 F^1_{USCGCON} (X) \subseteq F^1_{USCCON} (X).
\end{gather*}

Let
$(X,d)$ be a metric space.
 We
use $\bm{H}$ to denote the \emph{\textbf{Hausdorff distance}}
on
 $C(X)$ induced by $d$, i.e.,
\begin{equation} \label{hau}
\bm{H(U,V)}  =   \max\{H^{*}(U,V),\ H^{*}(V,U)\}
\end{equation}
for arbitrary $U,V\in C(X)$,
where
  $$
H^{*}(U,V)=\sup\limits_{u\in U}\,d\, (u,V) =\sup\limits_{u\in U}\inf\limits_{v\in
V}d\, (u,v).
$$

The Hausdorff semi-distance $H^*$ on $C(X)$ can be extended to $C(X) \cup \{\emptyset\}$
as follows:
\[  H^*(M_1, M_2) =\left\{
                \begin{array}{ll}
                 H^*(M_1, M_2) , & \hbox{if} \  M_1, M_2 \in C(X),
                   \\
                  +\infty, & \hbox{if} \  M_1\in C(X)\   \hbox{and} \ M_2=\emptyset,
                   \\
                  0, & \hbox{if}  \  M_1=\emptyset  \hbox{ and } M_2\in C(X)\cup \{\emptyset\}.
                \end{array}
              \right.
 \]
So the
Hausdorff distance
$H$
on $C(X)$ can be extended to $C(X) \cup \{\emptyset\}$
as follows:
\[  H(M_1, M_2) = H^*(M_1, M_2)\vee H^*(M_2, M_1)=\left\{
                \begin{array}{ll}
                 H(M_1, M_2) , & \hbox{if} \  M_1, M_2 \in C(X),
                   \\
                  +\infty, & \hbox{if} \  M_1=\emptyset \   \hbox{and} \ M_2 \in C(X),
                   \\
                  0, & \hbox{if}  \  M_1=M_2=\emptyset.
                \end{array}
              \right.
 \]

If
there is no confusion, we also use $H$ to denote the Hausdorff distance on $C(X\times [0,1])\cup \{\emptyset\}$ induced by $\overline{d}$.

In this paper,
for a metric space $(Y,\rho)$ and a subset $S$ in $Y$,
we still use $\rho$ to denote the induced metric on $S$ by $\rho$.

\begin{re}
{\rm

$\rho$ is said to be a \emph{metric} on $Y$ if $\rho$ is a function from $Y\times Y$ into $\mathbb{R}$
satisfying
positivity, symmetry and triangle inequality. At this time, $(Y, \rho)$ is said to be a metric space.

  $\rho$ is said to be an \emph{extended metric} on $Y$ if $\rho$ is a function from $Y\times Y$ into $\mathbb{R} \cup \{+\infty\} $
satisfying
positivity, symmetry and triangle inequality. At this time, $(Y, \rho)$ is said to be an extended metric space.

We can see that for arbitrary metric space $(X,d)$, the Hausdorff distance $H$ on $K(X)$ induced by $d$ is a metric.
So
the Hausdorff distance $H$ on $K(X\times [0,1])$ induced by $\overline{d}$ on $X\times [0,1]$
is a metric.

The Hausdorff distance $H$ on $C(X)$ induced by $d$ on $X$
is an extended metric, but probably not a metric,
because
$H(A,B)$ could be equal to $+\infty$ for certain metric space $X$ and $A, B \in C(X)$.

The Hausdorff distance $H$ on $C(X)\cup \{\emptyset\}$
is an extended metric. Since for each $x\in X$, $H(\{x\}, \emptyset)=+\infty$,
we have that the Hausdorff distance $H$ on $C(X)\cup \{\emptyset\}$ (respectively, $K(X)\cup \{\emptyset\}$) is not a metric.

Clearly, if $H$ on $C(X)$ induced by $d$
is not a metric, then $H$ on $C(X\times [0,1])$ induced by $\overline{d}$
is also not a metric.
So
the Hausdorff distance $H$ on $C(X\times [0,1])$ induced by $\overline{d}$ on $X\times [0,1]$
 is an extended metric but probably not a metric.

We can see that $H$ on $C(\mathbb{R}^m)$ is an extended metric but not a metric,
and then the same is $H$ on $C(\mathbb{R}^m\times [0,1])$.

When the Hausdorff distance $H$ is a metric, it is also called the Hausdorff metric.
When the Hausdorff distance $H$ is an extended metric, it is also called the Hausdorff extended metric.
In this paper, for simplicity,
 we refer to both the Hausdorff extended metric and the Hausdorff metric as the Hausdorff metric, except 
for the use of the term Hausdorff extended metric 
in this remark.

}
\end{re}

For
$u\in F(X)$,
define
\begin{gather*}
{\rm end}\, u:= \{ (x, t)\in  X \times [0,1]: u(x) \geq t\},
\\
{\rm send}\, u:= \{ (x, t)\in  X \times [0,1]: u(x) \geq t\} \cap  ([u]_0 \times [0,1]).
\end{gather*}
 $
{\rm end}\, u$ and ${\rm send}\, u$
 are called the endograph and the sendograph of $u$, respectively.

\begin{re}\label{bcp}
{\rm
Let $u \in  F(X)$. The following conditions (\romannumeral1)-(\romannumeral3)
are
equivalent:
\\
(\romannumeral1) \ $u\in F_{USC} (X)$;
\\
(\romannumeral2) ${\rm end}\, u$ is closed in $(X\times [0,1],   \overline{d})$;
\\
(\romannumeral3) ${\rm send}\, u$ is closed in $(X\times [0,1],  \overline{d})$.

(\romannumeral1)$\Rightarrow$(\romannumeral2). Assume that (\romannumeral1) is true.
To show that (\romannumeral2) is true, let $\{(x_n, \al_n)\}$ be a sequence
in ${\rm end}\, u$ which converges to $(x,\al)$ in $X\times [0,1]$, we only need to show that $(x,\al) \in {\rm end}\, u$.
Since
 $u$ is upper semi-continuous, then
$u(x) \geq \limsup_{n\to\infty} u(x_n) \geq \lim_{n\to\infty} \al_n = \al$. Thus $(x,\al) \in {\rm end}\, u$.
So (\romannumeral2) is true.

(\romannumeral2)$\Rightarrow$(\romannumeral3).
Assume that (\romannumeral2) is true.
Note that $[u]_0 \times [0,1]$ is closed in $X\times [0,1]$,
then
${\rm send}\, u =  {\rm end}\, u  \cap ([u]_0 \times [0,1])$ is closed in $X\times [0,1]$.
So
(\romannumeral3) is true.

(\romannumeral3)$\Rightarrow$(\romannumeral1).
 Assume that (\romannumeral3) is true.
To show that (\romannumeral1) is true,
let $\al\in [0,1]$ and suppose that $\{x_n\}$ is
 a sequence in $[u]_\al$ which converges to $x$ in $X$,
 we only need to show that $x\in [u]_\al$.
Note that
$\{(x_n, \al)\}$ converges to $(x,\al)$ in $X\times[0,1]$,
and that
the sequence $\{(x_n,\al)\}$ is in ${\rm send}\, u$.
Hence from the closedness of ${\rm send}\, u$, it follows that $(x,\al) \in {\rm send}\, u$,
which means that $x\in [u]_\al$.
So
(\romannumeral1) is true.

}
 \end{re}

Let $u\in F(X)$. Clearly $X\times\{0\} \subseteq {\rm end}\, u$. So
${\rm end}\, u \not=\emptyset$. Thus
by Remark \ref{bcp}, $u\in F_{USC}(X)$ if and only if
${\rm end}\, u\in C(X \times [0,1])$.

We can see that the conditions (\rmn1) $u=\emptyset_{F(X)}$,
(\rmn2) ${\rm send}\, u = \emptyset$, and (\rmn3) ${\rm end}\, u  = X\times\{0\}$,
are equivalent.

Kloeden \cite{kloeden2} introduced the \emph{\textbf{endograph metric}} \bm{$H_{\rm end}$}.
For $u,v \in F_{USC}(X)$,
\begin{gather*}
\bm{  H_{\rm end}(u,v)    }: =  H({\rm end}\, u,  {\rm end}\, v ),
  \end{gather*}
where
 $H$
is
the Hausdorff
metric on $C(X \times [0,1])$ induced by $\overline{d}$ on $X \times [0,1]$.

Rojas-Medar and Rom\'{a}n-Flores \cite{rojas} introduced
the \bm{$\Gamma$}\emph{\textbf{-convergence}} of a sequence
of upper semi-continuous fuzzy sets based on the Kuratowski convergence
of a sequence of sets in a metric space.

Let $(X,d)$ be a metric space.
Let $C$ be a set in $X$ and
$\{C_n\}$ a sequence of sets in $X$.
 $\{C_n\}$ is said to \emph{\textbf{Kuratowski converge}} to
$C$ according to $(X,d)$, if
$$
C
=
\liminf_{n\rightarrow \infty} C_{n}
=
\limsup_{n\rightarrow \infty} C_{n},
$$
where
\begin{gather*}
\liminf_{n\rightarrow \infty} C_{n}
 =
 \{x\in X: \  x=\lim\limits_{n\rightarrow \infty}x_{n},    x_{n}\in C_{n}\},
\\
\limsup_{n\rightarrow \infty} C_{n}
=
\{
 x\in X : \
 x=\lim\limits_{j\rightarrow \infty}x_{n_{j}},x_{n_{j}}\in C_{n_j}
\}
 =
 \bigcap\limits_{n=1}^{\infty}   \overline{   \bigcup\limits_{m\geq n}C_{m}    }.
\end{gather*}
In this case, we'll write
\bm{  $C=\lim^{(K)}_{n\to\infty}C_n    $ } according to $(X,d)$.
If
there is no confusion, we will not emphasize
the metric space
$(X,d)$ and write  $\{C_n\}$ \emph{\textbf{Kuratowski converges}} to
$C$ or \bm{  $C=\lim^{(K)}_{n\to\infty}C_n    $ } for simplicity.

\begin{re} \label{ksc}
{\rm
 Theorem 5.2.10 in \cite{beer} pointed out that, in a first countable Hausdorff topological space, a sequence of sets is Kuratowski convergent is
equivalent to this sequence is Fell topology convergent.
A metric space is of course a first countable Hausdorff topological space.

Definition 3.1.4 in \cite{kle} gives the definitions of
$\liminf C_{n}$, $\limsup C_{n}$
 and
$\lim C_{n}$
for
 a net of subsets $\{C_n, n\in D\}$ in a topological space.
When $\{C_n, n=1,2,\ldots\}$ is
 a sequence of subsets of a metric space,
$\liminf C_{n}$, $\limsup C_{n}$
 and
$\lim C_{n}$
according to Definition 3.1.4 in \cite{kle}
are
$\liminf_{n\rightarrow \infty} C_{n}$, $\limsup_{n\rightarrow \infty} C_{n}$
and $\lim^{(K)}_{n\to\infty}C_n  $
according to
 the above definitions, respectively.

  }
\end{re}

let $u$, $u_n$, $n=1,2,\ldots$, be fuzzy sets in $F(X)$.
   $\{u_n\}$ is said to \bm{$\Gamma$}-\emph{\textbf{converge}}
  to
  $u$, denoted by \bm{$u = \lim_{n\to \infty}^{(\Gamma)}  u_n$},
  if
  ${\rm end}\, u= \lim_{n\to \infty}^{(K)}  {\rm end}\, u_n$
according to
$(X\times [0,1], \overline{d})$.

The following Theorem \ref{hkg}
is an already known conclusion, which is useful in this paper.
We could not find the original file that gave Theorem \ref{hkg}.
 Theorem \ref{hkg} can be shown in a similar fashion to Theorem 4.1 in \cite{huang}.

\begin{tm} \label{hkg}
Suppose that $C$ and $C_n$, $n=1,2,\ldots$ are closed sets in $X$. Then $H(C_n, C) \to 0$ as $n\to\infty$ implies $\lim_{n\to \infty}^{(K)} C_n \, =C$.
\end{tm}

\begin{re} \label{hmr}
  {\rm

Theorem \ref{hkg} implies that
for a sequence
$\{u_n\}$ in $F_{USC}(X)$ and an element $u$ in $F_{USC}(X)$,
if
$H_{\rm end} (u_n, u) \to 0$ as $n \to \infty$, then $\lim_{n\to\infty}^{(\Gamma)} u_n = u$.
However,
the converse is false. See Example 4.1 in \cite{huang} and Proposition \ref{csm}.

}
\end{re}

\begin{re} \label{gby}
 {\rm
   Let $\{u_n\}$ be a sequence in $F(X)$ and
let $\{v_n\}$ be a subsequence of $\{u_n\}$. We can see that
$$\liminf_{n\to\infty} {\rm end}\, u_n \subseteq \liminf_{n\to\infty} {\rm end}\, v_n \subseteq \limsup_{n\to\infty} {\rm end}\, v_n  \subseteq \limsup_{n\to\infty} {\rm end}\, u_n.$$
So
if there is a $u\in F(X)$ with $\lim_{n\to\infty}^{(\Gamma)} u_n = u$, then $\lim_{n\to\infty}^{(\Gamma)} v_n = u$.

Clearly, $\lim_{n\to\infty}^{(\Gamma)} v_n = u$
does not necessarily imply
 that $\lim_{n\to\infty}^{(\Gamma)} u_n = u$. A simple example is given below.

 For $n=1,2,\ldots$, let
 $v_n= \widehat{1}_{F(\mathbb{R})}$.
 For $n=1,2,\ldots$, define $u_n$ by
 \[
u_n=\left\{
        \begin{array}{ll}
         \widehat{1}_{F(\mathbb{R})}, & n \mbox{ is odd,}\\
          \widehat{3 }_{F(\mathbb{R})}, & n \mbox{ is even.}
        \end{array}
      \right.
 \]
 Then $\{u_n\} \subseteq F_{USCB} (\mathbb{R})$ and $\{v_n\}$ is a subsequence of $\{u_n\}$. We can see that
 $\lim_{n\to\infty}^{(\Gamma)} v_n = \widehat{1}_{F(\mathbb{R})}$. However
 $\lim_{n\to\infty}^{(\Gamma)} u_n $ does not exist
because
 $$\liminf_{n\to\infty} {\rm end}\, u_n = \mathbb{R}\times \{0\}
\subsetneqq {\rm end}\, {\widehat{1}_{F(\mathbb{R})}} \cup {\rm end}\, {\widehat{3}_{F(\mathbb{R})}}
=\limsup_{n\to\infty} {\rm end}\, u_n . $$

 }
 \end{re}

\begin{re}\label{cmu}
  {\rm
Another metric $d'$ on $X \times [0,1]$ is defined
as follows: for each $(x,\al)$, $(y, \beta)$ in $X\times [0,1]$,
$d'((x,\al), (y, \beta)) = \max\{ d(x,y), |\al-\beta| \}$.
Based on $d'$, the corresponding endograph metric on $F_{USC}(X)$ and $\Gamma$-convergence on $F(X)$ can be derived. We denote them by $H'_{\rm end}$ and $\Gamma'$-convergence, respectively.

It's known that (\rmn1)
$d' \leq \overline{d}  \leq 2 d'$ on $X\times [0,1]$,
 (\rmn2)
for each $u,v\in F_{USC}(X)$, $H'_{\rm end}(u,v) \leq H_{\rm end}(u,v) \leq 2H'_{\rm end}(u,v)$, and
(\rmn3)
the $\Gamma$-convergence is equivalent to the
 $\Gamma'$-convergence on $F(X)$ ((\rmn2) and (\rmn3) follows immediately from (\rmn1)).
So
it is easy to see what conclusions in this paper remain true
if ``$\overline{d}$'', ``$H_{\rm end}$'' and ``$\Gamma$-convergence''
are replaced by ``$d'$'', ``$H_{\rm end}'$'' and ``$\Gamma'$-convergence'', respectively.
 Lemma \ref{tbfegnum}, Theorems \ref{rcfegnum}, \ref{tbfegnu},
  \ref{cfegum}, Propositions \ref{hger}, \ref{hge}, \ref{csm}
and so on are such conclusions.
There is no need to list these conclusions  one by one, because it is easy to see.

We use $H'^*$ and $H'$
to denote the Hausdorff semi-distance and the Hausdorff distance on $C(X\times [0,1])\cup\{\emptyset\}$ induced by $d'$, respectively.
It's known that $H'^* \leq H^* \leq 2 H'^*$ and
$H' \leq H \leq 2 H'$ on $C(X\times [0,1])\cup \{\emptyset\}$.
It's easy to see what contents in this paper
remain true
if ``$\overline{d}$'', ``$H_{\rm end}$'', ``$\Gamma$-convergence'', $H^*$ and $H$ on $C(X\times [0,1])\cup\{\emptyset\}$
are replaced by ``$d'$'', ``$H_{\rm end}'$'', ``$\Gamma'$-convergence'', $H'^*$ and $H'$ on $C(X\times [0,1])\cup\{\emptyset\}$, respectively.

The Hausdorff distance $H'$ on $C(X\times [0,1])\cup\{\emptyset\}$
is an extended metric. 
Note that for each $A,B\in C(X\times [0,1])\cup\{\emptyset\}$,
$H(A,B)=+\infty$ if and only if $H'(A,B)=+\infty$.
So for each subset $S$ of $C(X\times [0,1])\cup\{\emptyset\}$,
the restriction of $H$ to $S\times S$ is a metric on $S$
if and only if the restriction of $H'$ to $S\times S$ is a metric on $S$.
}
\end{re}

\section{Representation theorems for various kinds of fuzzy sets}
\label{rcem}

In this section, we give representation theorems for various kinds of fuzzy sets.
These representation theorems
are useful in this paper.

The following representation theorem should be a known conclusion.

\begin{tm} \label{rep}
Let $Y$ be a nonempty set.
If $u\in F(Y)$, then for all
$\al\in (0,1]$,
  $[u]_\al = \cap_{\beta<\al} [u]_\beta.$

Conversely,
suppose that
$\{v_\al: \al \in (0,1]\}$
is
a family of sets in $Y$ with $v_\al = \cap_{\beta<\al} v_\beta$ for all
$\al\in (0,1]$.
Define $u\in F(Y)$ by
$$u(x) := \sup\{  \al:  x\in v_\al    \} $$ for each $x\in Y$.
 Then $u$ is the unique fuzzy set in $Y$ satisfying that
 $[u]_\al = v_\al$ for all $\al\in (0,1]$.
\end{tm}

\begin{proof}
Let $u\in F(Y)$ and $\al\in (0,1]$. For each $x\in Y$,
$x\in [u]_\al \Leftrightarrow u(x) \geq \al \Leftrightarrow$ for each $\beta<\al$, $u(x) \geq \beta$
$\Leftrightarrow$ for each $\beta<\al$, $x\in [u]_\beta$.
So
 $[u]_\al = \cap_{\beta<\al} [u]_\beta.$

Conversely,
suppose that
$\{v_\al: \al \in (0,1]\}$
is
a family of sets in $Y$ with $v_\al = \cap_{\beta<\al} v_\beta$ for all
$\al\in (0,1]$.
Let $u\in F(Y)$ defined by
$$u(x) := \sup\{  \al:  x\in v_\al    \} $$ for each $x\in Y$.
Firstly, we show
that for each $\al\in (0,1]$,
$[u]_\al = v_\al$.
To do this, let $\al\in (0,1]$. We only need to verify that $[u]_\al \supseteq v_\al$ and $[u]_\al\subseteq v_\al$.

Let $x\in  v_\al $. Then clearly $u(x) \geq \al$, i.e.
$x\in [u]_\al$.
So $[u]_\al \supseteq v_\al$.

Let $x\in [u]_\al$. Then
$\sup\{ \beta:  x\in v_\beta    \} = u(x)\geq \al$.
Hence
there exists a sequence $\{ \beta_n, \ n=1,2,\ldots\}$ such that
$1\geq\beta_n \geq \al-1/n$ and $x\in v_{\beta_n}$.
Set $\gamma= \sup_{n=1}^{+\infty} \beta_n $. Then $1\geq \gamma\geq \alpha$
and thus $x\in \cap_{n=1}^{+\infty} v_{\beta_n} = v_\gamma \subseteq v_\al$.
So $[u]_\al\subseteq v_\al$.

Now we show the uniqueness of $u$. To do this, assume that
 $v$ is a fuzzy set in $Y$ satisfying that
 $[v]_\al = v_\al$ for all $\al\in (0,1]$. Then for each $x\in Y$,
$$
v(x) = \sup\{  \al:  x\in [v]_\al    \}  = \sup\{  \al:  x\in v_\al    \} = u(x).
$$
So $u=v$.

\end{proof}

\begin{re}
  {\rm

We can't find the original reference which gave Theorem \ref{rep}, so we give a proof here for the self-containing of this paper.
Theorem \ref{rep} and its proof are essentially the same as the Theorem 7.10 in P27 of chinaXiv:202110.00083v4
and its proof since
the uniqueness of $u$ is obvious.

}
\end{re}

From Theorem \ref{rep}, it follows immediately below representation theorems for
$F_{USC} (X)$, $F^1_{USC} (X)$, $F_{USCG} (X)$, $F_{CON} (X)$, $F_{USCB} (X)$,
and $F^1_{USCB} (X)$.

\begin{pp} \label{fus}\
  Let $(X,d)$ be a metric space.
If $u\in F_{USC} (X)$ (respectively, $u\in F^1_{USC} (X)$, $u\in F_{USCG} (X)$, $u\in F_{CON} (X)$), then
\\
(\romannumeral1) \ $[u]_\al\in C(X) \cup \{\emptyset\}$ (respectively, $[u]_\al\in C(X)$, $[u]_\al\in K(X) \cup \{\emptyset\}$, $[u]_\al$ is connected in $(X,d)$) for all $\al\in (0,1]$, and
\\
(\romannumeral2) \ $[u]_\al=\bigcap_{\beta<\al}[u]_\beta$ for all $\al\in (0,1]$.

Conversely, suppose that the family of sets $\{v_\al:\al\in (0,1]\}$ satisfies
conditions (\romannumeral1) and (\romannumeral2).
Define $u\in F(X)$ by
$u(x) := \sup\{  \al:  x\in v_\al    \} $ for each $x\in X$.
Then $u$ is the unique fuzzy set
in $X$
satisfying that $[u]_{\al}=v_\al$ for each $\al\in (0,1]$.
Moreover, $u\in F_{USC}(X)$ (respectively, $u\in F^1_{USC} (X)$, $u\in F_{USCG}(X)$, $u\in F_{CON}(X)$).

\end{pp}

\begin{proof}
  The proof is routine. We only show the case of $F_{USC} (X)$. The other cases can be verified similarly.

If $x\in F_{USC} (X)$, then clearly (\romannumeral1) is true.
From Theorem \ref{rep},
(\romannumeral2) is true.

Conversely, suppose that the family of sets $\{v_\al:\al\in (0,1]\}$ satisfies
conditions (\romannumeral1) and (\romannumeral2).
Define $u\in F(X)$ by
$u(x) := \sup\{  \al:  x\in v_\al    \} $ for each $x\in X$.
Then by Theorem \ref{rep}, $u$ is the unique fuzzy set
in $X$
satisfying that $[u]_{\al}=v_\al$ for each $\al\in (0,1]$.
Since $\{[u]_\al, \al\in (0,1] \}$ satisfies condition (\romannumeral1),
$u\in F_{USC}(X)$.

\end{proof}

\begin{pp} \label{fuscbchre}\
  Let $(X,d)$ be a metric space.
If $u \in F_{USCB} (X)$ (respectively, $u \in F^1_{USCB} (X)$), then
\\
(\romannumeral1) \ $[u]_\al\in K(X) \cup \{\emptyset\}$ (respectively, $[u]_\al\in K(X)$) for all $\al\in [0,1]$,
\\
(\romannumeral2) \ $[u]_\al=\bigcap_{\beta<\al}[u]_\beta$ for all $\al\in (0,1]$, and
\\
(\romannumeral3) \ $[u]_0=\overline{\bigcup_{\beta>0}[u]_\beta}$.

Conversely, suppose that the family of sets $\{v_\al:\al\in [0,1]\}$ satisfies
conditions (\romannumeral1) through (\romannumeral3).
Define $u\in F(X)$ by
$u(x) := \sup\{  \al:  x\in v_\al    \} $ for each $x\in X$.
Then $u$ is the unique fuzzy set
in $X$
satisfying that $[u]_{\al}=v_\al$ for each $\al\in [0,1]$.
Moreover, $u\in F_{USCB}(X)$
(respectively, $u \in F^1_{USCB} (X)$).

\end{pp}

\begin{proof}
    The proof is routine. We only show the case of $F_{USCB} (X)$. The case of $F^1_{USCB} (X)$ can be verified similarly.

If $x\in F_{USCB} (X)$, then clearly (\romannumeral1) is true.
By Theorem \ref{rep},
(\romannumeral2) is true.
From the definition of $[u]_0$, (\romannumeral3) is true.

Conversely, suppose that the family of sets $\{v_\al:\al\in [0,1]\}$ satisfies
conditions (\romannumeral1) through (\romannumeral3).
Define $u\in F(X)$ by
$u(x) := \sup\{  \al:  x\in v_\al    \} $ for each $x\in X$.
Then by Theorem \ref{rep}, $u$ is the unique fuzzy set
in $X$
satisfying that $[u]_{\al}=v_\al$ for each $\al\in (0,1]$.
Clearly $[u]_0 = \overline{\bigcup_{\beta>0}[u]_\beta} = \overline{\bigcup_{\beta>0}v_\beta} = v_0$.
Since $\{[u]_\al, \al\in [0,1] \}$ satisfies condition (\romannumeral1), $u\in F_{USCB}(X)$.

\end{proof}

Similarly,
we can obtain the
representation theorems for $F_{USCCON}(X)$, $F_{USCGCON}(X)$, $F^1_{USCCON}(X)$, etc.

Based on these representation theorems, we can define a fuzzy set or a certain type fuzzy set
by giving the family of its $\al$-cuts.
In the sequel,
we will directly point out that what we defined is a fuzzy set or a certain type fuzzy set without saying which  representation theorem is used since it is easy to see.

\section{Basic properties
of $X\times[0,1]$ \label{bpxe}}

In this section, we give some basic properties
of $X\times[0,1]$, which are useful in this paper.

Below conclusion is well-known.

  Let $(Y,\tau_1)$ and $(Z, \tau_2)$ be topological space, and let
$f$ be a continuous function from $(Y,\tau_1)$ to $(Z, \tau_2)$.
 If
$A$ is compact in $(Y,\tau_1)$, then $f(A)$ is compact in $(Z, \tau_2)$.

Let \bm{$p_X$} be the projection mapping from $(X\times [0,1], \overline{d})$ to $(X, d)$;
that is,
$$\mbox{for each $(x,\al) \in  X\times [0,1]$,\ $p_X (x,\al) =x$.}$$
Clearly for $(x,\al), (y,\beta)\in  X\times [0,1]$,
$d(p_X(x,\al), p_X(y,\beta)) = d(x,y) \leq \overline{d} ( (x,\al), (y,\beta)) $.
So $p_X$ is continuous.

For $u \subseteq  X \times [0,1]$ and $\al\in [0,1]$,
define
 $\bm{\langle u \rangle_\al} := \{ x:   (x,\al) \in u\}$.

Let $u \subseteq  X \times [0,1]$. Then
$p_X(u) = \cup_{\al\in [0,1]} \langle u \rangle_\al$.
Clearly
$$
u\not=\emptyset \Leftrightarrow p_X(u) \not= \emptyset
\Leftrightarrow \cup_{\al\in [0,1]} \langle u \rangle_\al\not=\emptyset.
$$
If $u \in  K(X \times [0,1])$, then $p_X(u) \not= \emptyset$
and $p_X(u)$ is compact in $(X,d)$; that is,
$p_X(u)\in K(X)$.

\begin{pp}
  \label{bpe}
Let $u$ be a subset of $X\times [0,1]$.
\\
(\romannumeral1)
If $u\in C(X\times [0,1])$, then $\langle u \rangle_\al \in C(X) \cup \{\emptyset\}$ for all $\al\in [0,1]$.
\\
(\romannumeral2) $u\in K(X\times [0,1])$ if and only if $u\in C(X\times [0,1])$
and
$\bigcup_{\al\in [0,1]} \langle u \rangle_\al \in K(X)$.
\\
(\romannumeral3)
If $u\in K(X\times [0,1])$, then $\langle u \rangle_\al \in K(X) \cup \{\emptyset\}$ for all $\al\in [0,1]$.
\\
(\romannumeral4)
Even if $u\not= \emptyset$, the converse of (\rmn1) and the converse of (\rmn3) are false.
\end{pp}

\begin{proof}

To show (\rmn1), assume that $u\in C(X\times [0,1])$ and $\al\in [0,1]$. To show that $\langle u \rangle_\al$ is closed in $(X,d)$, let $\{x_n\}$ be a sequence
in $\langle u \rangle_\al$ such that $\{x_n\}$ converges to $x$ in $(X,d)$.
We only need
to show
that $x\in \langle u \rangle_\al$.

Note that $\{(x_n,\al)\}$ converges to $(x,\al)$ in $(X\times [0,1], \overline{d})$.
Thus $(x,\al) \in u$. This means that $x\in \langle u \rangle_\al$.
So $\langle u \rangle_\al$ is closed in $(X,d)$; that is, $\langle u \rangle_\al \in C(X) \cup \{\emptyset\}$.
Thus (\rmn1) is true.

Next we show (\rmn2).
\emph{Necessity}.
Assume that $u\in K(X\times [0,1])$.
Then $\bigcup_{\al\in [0,1]} \langle u \rangle_\al = p_X(u)\in K(X)$.
Clearly $u\in C(X\times [0,1])$.
 So the necessity is proved.

\emph{Sufficiency}.
Assume that $u\in C(X\times [0,1])$
and
$\bigcup_{\al\in [0,1]} \langle u \rangle_\al \in K(X)$.
As $[0,1] \in K([0,1])$, we have that
$\bigcup_{\al\in [0,1]} \langle u \rangle_\al \times [0,1]\in K(X\times [0,1])$.
Note that $u$ is a nonempty closed subset of $\bigcup_{\al\in [0,1]} \langle u \rangle_\al \times [0,1]$.
Thus
$u \in K(X\times [0,1])$. So the sufficiency is proved and (\romannumeral2) is true.

To show (\rmn3), assume that $u\in K(X\times [0,1])$ and $\al\in [0,1]$.
Then by (\rmn1) and (\rmn2),
$\bigcup_{\al\in [0,1]} \langle u \rangle_\al \in K(X)$
and
$\langle u \rangle_\al \in C(X)\cup \{\emptyset\}$.
Thus
$\langle u \rangle_\al \in K(X)\cup \{\emptyset\}$
as $\langle u \rangle_\al$ is a closed subset of $\bigcup_{\al\in [0,1]} \langle u\rangle_\al$.
So (\rmn3) is true.

Finally
we show (\rmn4),
define a nonempty subset $u$ of $\mathbb{R}\times[0,1]$ by putting
\[
\langle u \rangle_\al
=\left\{
   \begin{array}{ll}
    [0,1] , & \al\in [0,1/2), \\
 \empty [0,1/2], & \al\in [1/2, 1].
   \end{array}
 \right.
\]
Then for each $\al\in [0,1]$,
$\langle u \rangle_\al \in K(\mathbb{R})\subset C(\mathbb{R})$.

Let $x$ be the point $(1,\frac{1}{2})$ in $\mathbb{R}\times [0,1]$.
For $n=1,2,\ldots$,
let
$x_n$ be the point $(1,\frac{1}{2}-\frac{1}{2n})$. Then $\{x_n\}$ is a sequence in $u$.
Denote the Euclidean metric on $\mathbb{R}$ by $\rho$.
We can see that
$\overline{\rho} (x_n, x) \to 0$.
However $x\notin u$ and so $u\notin C(\mathbb{R}\times [0,1])$.
This indicates that even if $u\not=\emptyset$, the converse of (\rmn1) and the converse of (\rmn3) are not true.
So (\rmn4) is proved.
\end{proof}

Let $c_0, c_1\in[0,1]$ with $0\leq c_0 \leq c_1 \leq 1$ and
let $u$ be a subset of $X\times [c_0, c_1]$.
The following conclusions (\rmn1) and (\rmn2) are obvious. We will use them without mentioning them.

(\rmn1) Clearly, $u\in C(X\times [c_0, c_1])$ if and only if $u\in C(X\times [0,1])$;
$u\in K(X\times [c_0, c_1])$ if and only if $u\in K(X\times [0,1])$.

(\rmn2) As $\langle u \rangle_\al = \emptyset$ for all $\al\in [0,1] \setminus [c_0, c_1]$, it follows that
 $\langle u \rangle_\al \in C(X) \cup \{\emptyset\}$ for all $\al\in [c_0, c_1]$
if and only if
 $\langle u \rangle_\al \in C(X) \cup \{\emptyset\}$ for all $\al\in [0,1]$;
$\langle u \rangle_\al \in K(X) \cup \{\emptyset\}$ for all $\al\in [c_0, c_1]$
if and only if
 $\langle u \rangle_\al \in K(X) \cup \{\emptyset\}$ for all $\al\in [0,1]$.

For $u \subseteq X \times [0,1] $,
define
$\bm{S_u} := \sup\{ \al:  (x,\al) \in u  \}$.
We call $S_u$ the height of $u$.
In this paper we assume that \bm{$\sup\emptyset = 0$}.

\begin{pp}\label{bpu}
  Let $u$ be a subset of $X\times [0,1]$ and $c\in [0,S_u]$. Suppose the following conditions.
\\
(\rmn1)  $u\cap (X\times [c, 1]) \in K(X\times [0,1])$.
\\
(\rmn2)  $ \cup_{\al\in [c,1]}\langle u \rangle_{\al} \in K(X)$ and $u\cap (X\times [c, 1]) \in C(X\times [0,1])$.
\\
(\rmn3) $\langle u \rangle_{S_u} \in K(X)$.
\\
 Then (\rmn1)$\Leftrightarrow$(\rmn2)$\Rightarrow$(\rmn3).
\end{pp}

\begin{proof}
Set $v:= u\cap (X\times [c, 1]) $. By Proposition \ref{bpe} (\rmn2),
 $v \in K(X\times [0,1])$
if and only if
$v\in C(X\times [0,1])$ and $\cup_{\al\in [0,1]}\langle v \rangle_{\al}   =\cup_{\al\in [c,1]}\langle u \rangle_{\al} \in K(X)$.
So (\rmn1)$\Leftrightarrow$(\rmn2).

To complete the proof, we only need to prove (\rmn1)$\Rightarrow$(\rmn3).
Suppose that (\rmn1) is true. By
 Proposition \ref{bpe} (\rmn3),
 $\langle u \rangle_{S_u}\in K(X) \cup \{\emptyset\}$.
So to show (\rmn3), we only need to show that
$\langle u \rangle_{S_u} \not= \emptyset$. We prove by contradiction.

Assume that $\langle u \rangle_{S_u} = \emptyset$.
Then $c<S_u$ and there is a sequence $\{(x_n, \al_n)\}$ in $u\cap (X\times [c, S_u))$ such that $\al_n\to S_u$ as $n\to\infty$.
From (\rmn1), there is a subsequence $\{(x_{n_k}, \al_{n_k})\}$
of $\{(x_n, \al_n)\}$ which converges to
$(x,\al)$ in $(u\cap (X\times [c, 1]), \overline{d})$.
Hence $\al=S_u$ and then $\langle u \rangle_{S_u} \not= \emptyset$.
This is a contradiction. Thus $\langle u \rangle_{S_u} \not= \emptyset$,
 and (\rmn3) is true. So (\rmn1)$\Rightarrow$(\rmn3). This completes the proof.

The conclusion that (\rmn1)$\Leftrightarrow$(\rmn2)
is equivalent to Proposition \ref{bpe} (\rmn2).
 Above we show that Proposition \ref{bpe} (\rmn2) implies
that (\rmn1)$\Leftrightarrow$(\rmn2). Conversely,
set $c=0$. Then from (\rmn1)$\Leftrightarrow$(\rmn2), we obtain
 Proposition \ref{bpe} (\rmn2).

\end{proof}

\begin{re} \label{bre}
  {\rm The conclusions in this remark are easy to obtain. We will often use them without citation.

Let $c_0, c_1, r_0,r_1 \in [0,1]$ with $r_0\leq r_1$ and $c_0\leq c_1$.
Let $u$ be a subset of $X\times [c_0, c_1]$.
\\
(\rmn1)
Clearly for each
$\al\in [0,1]  \setminus [c_0, S_u]$,
$ \langle u \rangle_\al = \emptyset$.
\\
(\rmn2)
Suppose the following conditions
(\rmn2-1) $u\in C(X\times [0,1])$,
(\rmn2-2)
 $u\cap (X\times [r_0,r_1])\in C(X\times [0,1])$,
(\rmn2-3) $u\in K(X\times [0,1])$,
(\rmn2-4)
 $u\cap (X\times [r_0,r_1])\in K(X\times [0,1])$.
Then
(\rmn2-1)$\Rightarrow$(\rmn2-2) and (\rmn2-3)$\Rightarrow$(\rmn2-4).

As $X\times [r_0,r_1]\in C(X\times [0,1])$, we have that (\rmn2-1)$\Rightarrow$(\rmn2-2). Assume that (\rmn2-3) is true. Then $u\cap (X\times [r_0,r_1])\in C(X\times [0,1])$. This means that $u\cap (X\times [r_0,r_1])$ is a closed subset of $u$. So (\rmn2-3)$\Rightarrow$(\rmn2-4).

Clearly,
(\rmn2-2) does not necessarily  imply (\rmn2-1); (\rmn2-4) does not necessarily imply (\rmn2-3).

}
\end{re}

\begin{pp} \label{bpue}
Let $c_0, c_1\in[0,1]$ with $0\leq c_0 \leq c_1 \leq 1$. Let $u$ be a nonempty subset of $X\times [c_0, c_1]$ satisfying that
$\langle u\rangle_\xi \subseteq \langle u\rangle_\eta$ for all $\xi,\eta \in [c_0, c_1]$ with $\eta\leq \xi$.
\\
(\rmn1) $\langle u \rangle_{c_0} =\cup_{\gamma\in [c_0, c_1]} \langle u\rangle_\gamma =p_X(u) \not= \emptyset$.
\\
(\rmn2) Suppose that $u\in C(X\times [0,1])$ and that there is a $c_2\in [c_0,S_u]$ with $ \langle u \rangle_{c_2} \in K(X)$.
(\rmn2-1) Let $\al\in [0,1]$. Then
$ \langle u \rangle_\al \not= \emptyset$
if and only if
$\al\in [c_0, S_u]$;
(\rmn2-2) for each $\al\in [c_2, S_u]$, $\langle u \rangle_\al \in K(X)$.
\\
(\rmn3) Suppose that $u\in K(X\times [0,1])$. Then
 for each $\al\in [c_0, S_u]$, $ \langle u \rangle_\al \in K(X)$.
\\
(\rmn4) Suppose that there is a $c_2\in [c_0,S_u]$
such that $u\cap (X\times [c_2, 1])\in K(X\times [0,1])$.
 (\rmn4-1) Let $\al\in [0,1]$. Then
$ \langle u \rangle_\al \not= \emptyset$
if and only if
$\al\in [c_0, S_u]$;
(\rmn4-2) for each $\al\in [c_2, S_u]$, $\langle u \rangle_\al \in K(X)$.

\end{pp}

\begin{proof}
Clearly (\rmn1) is true.
Next we show (\rmn4). Set $v:=u\cap (X\times [c_2, 1])$. As $v\in K(X\times [0,1])$,
by Proposition
\ref{bpu}, $\langle u \rangle_{S_u} \in K(X)$.
So $\langle u \rangle_{S_u} \not= \emptyset$ and hence
for each $\al\in [c_0, S_u]$,
 $\langle u \rangle_{\alpha}\not=\emptyset$; whence, by
Remark \ref{bre}(\rmn1),
 (\rmn4-1) is true.
By
 Proposition \ref{bpe}(\rmn3), for each $\al\in [c_2, S_u]$,
 $\langle u \rangle_{\alpha}=\langle v \rangle_{\alpha} \in K(X) \cup \{\emptyset\}$.
Also by (\rmn4-1),
 for each $\al\in [c_2, S_u]$,
 $\langle u \rangle_{\alpha} \not= \emptyset$.
Thus (\rmn4-2) is true.
So (\rmn4) is proved.

Now we show (\rmn2).
As $u\in C(X\times [0,1])$, clearly
$u\cap (X\times [c_2, 1])\in C(X\times [0,1])$.
Since
$ \langle u \rangle_{c_2} = \cup_{\al\in [c_2,1]}\langle u \rangle_{\al} \in K(X)$,
by Proposition \ref{bpu},
 $u\cap (X\times [c_2, 1])\in K(X\times [0,1])$.
Thus (\rmn2) follows from (\rmn4).

Finally we show (\rmn3).
By Proposition \ref{bpe}(\rmn2),
$u\in K(X\times [0,1])$ implies that
$u\in C(X\times [0,1])$ and $\langle u \rangle_{c_0}  = \cup_{\al\in [0,1]} \langle u \rangle_\al \in K(X)$.
So by (\rmn2-2), (\rmn3) is true.
This completes the proof.

We can see that (\rmn3) can also be implied by (\rmn4).
Indeed (\rmn2), (\rmn3) and (\rmn4) are equivalent to each other.
Above we show that (\rmn4) implies (\rmn2),
 (\rmn2) implies (\rmn3).
Below we show that
(\rmn3) implies (\rmn4).

If $u\cap (X\times [c_2, 1])\in K(X\times [0,1])$,
then
by (\rmn3), (\rmn4-2) is true.
So $ \langle u \rangle_{S_u} \not= \emptyset$
and from this we can deduce that (iv-1) is true.
Thus (\rmn3) implies (\rmn4).

\end{proof}

Clearly for each $r\in \mathbb{R}$,
 $(r,r]=\{x: r<x\leq r\}=\emptyset$.

\begin{pp}\label{bpec}
Let $c_0, c_1\in[0,1]$ with $0\leq c_0 \leq c_1 \leq 1$. Let $u$ be a nonempty subset of $X\times [c_0, c_1]$ satisfying that
for each $\al\in (c_0, c_1]$, $\langle u\rangle_\al = \bigcap_{\beta\in [c_0, \al)} \langle u\rangle_\beta$. Then
\\
(\rmn1) For each $\xi,\eta \in [c_0, c_1]$ with $\eta\leq \xi$, $\langle u \rangle_\xi\subseteq \langle u \rangle_\eta$.
 And $\langle u \rangle_{c_0} \not= \emptyset$.
\\
(\rmn2)
$u\in C(X\times [c_0,c_1])$ if and only if $\langle u \rangle_\al \in C(X) \cup \{\emptyset\}$ for all $\al\in [c_0,c_1]$.
\\
(\rmn3) $u\in K(X\times [c_0,c_1])$ if and only if $\langle u \rangle_\al \in K(X) \cup \{\emptyset\}$ for all $\al\in [c_0,c_1]$.
\end{pp}

\begin{proof}
First we show (\rmn1).
 Let $\xi,\eta \in [c_0, c_1]$ with $\eta\leq \xi$.
As $\langle u\rangle_\xi = \bigcap_{\gamma\in [c_0, \xi)} \langle u\rangle_\gamma$, it follows that $\langle u \rangle_\xi\subseteq \langle u \rangle_\eta$. Thus by Proposition \ref{bpue}(\rmn1), $\langle u \rangle_{c_0} \not= \emptyset$. So (\rmn1) is true.

To show (\rmn2), by Proposition \ref{bpe}(\rmn1), we only need to verify that
if
for each $\al\in [c_0,c_1]$,
$\langle u \rangle_\al \in C(X) \cup \{\emptyset\}$, then $u\in C(X\times [c_0, c_1])$.
To this end,
let $\{(x_n, \al_n)\}$ be a sequence
in $u$ which converges to an element $(x,\al)$ in $(X\times [c_0, c_1],  \overline{d})$.
It suffices to show that
$(x,\al)\in u$.

Note that $d(x_n, x) \to 0$ and $|\al_n-\al|\to 0$.
Assume that $\al=c_0$.
Notice that $\{x_n\}$ is a sequence in $\langle u\rangle_{c_0}$. Thus
$x\in \langle u\rangle_{c_0}$ as $\langle u\rangle_{c_0}\in C(X)$.
This means that $(x,\al)=(x,c_0)\in u$.

Assume that $\al>c_0$.
Given $\beta\in [c_0,\al)$. Then there is an $N
\in \mathbb{N}$ such that $\al_n>\beta$ for all $n\geq N$.
This implies that $x_n\in \langle u \rangle_{\beta}$ for all $n\geq N$.
Thus $x\in \langle u \rangle_{\beta}$ as $d(x_n, x)\to 0$ and $\langle u \rangle_{\beta}$ is closed in $X$.
  Hence
$x\in \bigcap_{\beta\in [c_0,\al)} \langle u \rangle_{\beta} = \langle u \rangle_\al$.
This means that $(x,\al)\in u$. So (\rmn2) is proved.

Now we show (\rmn3).
By Proposition \ref{bpe}(\rmn3), $u\in K(X\times [c_0,c_1])$ implies that $\langle u \rangle_\al \in K(X) \cup \{\emptyset\}$ for all $\al\in [c_0,c_1]$.

Assume that
$\langle u \rangle_\al \in K(X) \cup \{\emptyset\}$ for all $\al\in [c_0,c_1]$.
Then by (\rmn2), $u\in C(X\times [c_0,c_1])$.
Note that
$\bigcup_{\al\in [0,1]} \langle u \rangle_\al = \langle u \rangle_{c_0} \in K(X)$.
Thus by Proposition \ref{bpe}(\rmn2),
$u\in K(X\times [c_0,c_1])$.
So (\rmn3) is proved.

\end{proof}

 Let $u$ be a subset of $X\times [0, 1]$.
Clearly
 $u=\emptyset$ if and only if for each $\al\in [0,1]$, $\langle u \rangle_\al = \emptyset$. So
Proposition \ref{bpec}(\rmn1)
is equivalent to Corollary \ref{bpecun}(\rmn1);
Proposition \ref{bpec}(\rmn2) is equivalent to Corollary \ref{bpecun}(\rmn2);
 Proposition \ref{bpec}(\rmn3) is equivalent to Corollary \ref{bpecun}(\rmn3).

\begin{tl}
\label{bpecun}
Let $c_0, c_1\in[0,1]$ with $0\leq c_0 \leq c_1 \leq 1$. Let $u$ be a subset of $X\times [c_0, c_1]$ satisfying that
for each $\al\in (c_0, c_1]$, $\langle u\rangle_\al = \bigcap_{\beta\in [c_0, \al)} \langle u\rangle_\beta$. Then
\\
(\rmn1) For each $\xi,\eta \in [c_0, c_1]$ with $\eta\leq \xi$, $\langle u \rangle_\xi\subseteq \langle u \rangle_\eta$.
\\
(\rmn2)
$u\in C(X\times [0, 1]) \cup \{\emptyset\}$ if and only if $\langle u \rangle_\al \in C(X) \cup \{\emptyset\}$ for all $\al\in [0,1]$.
\\
(\rmn3) $u\in K(X\times [0,1]) \cup \{\emptyset\}$ if and only if $\langle u \rangle_\al \in K(X) \cup \{\emptyset\}$ for all $\al\in [0,1]$.
  \end{tl}

We use \bm{$P_{USC}(X)$} to denote the family of all subsets of $X\times[0,1]$
which satisfies that (\rmn1) for each
  $\al\in [0,1]$, $\langle u \rangle_\al \in C(X)\cup \{\emptyset\}$,
and (\rmn2) for each
  $\al\in (0,1]$, $\langle u\rangle_\al = \bigcap_{\beta<\al} \langle u\rangle_\beta$.

We use \bm{$P_{USCB}(X)$} 
to denote the family of all subsets of $X\times[0,1]$
which satisfies that (\rmn1) for each
  $\al\in [0,1]$, $\langle u \rangle_\al \in K(X)\cup \{\emptyset\}$,
and (\rmn2) for each
  $\al\in (0,1]$, $\langle u\rangle_\al = \bigcap_{\beta<\al} \langle u\rangle_\beta$.

Clearly
$P_{USCB}(X) 
=\{u\in P_{USC}(X):
   \langle u \rangle_\al \in K(X)\cup \{\emptyset\} \mbox{ for all } \al\in [0,1]\}
= 
\{u\in P_{USC}(X): \langle u \rangle_0 \in K(X)\cup \{\emptyset\} \}$.

Obviously $\emptyset \in P_{USCB}(X)\subseteq P_{USC}(X)$.
By Corollary \ref{bpecun},
\begin{gather}
P_{USC}(X)=\left\{  u \in C(X\times [0,1])\cup \{\emptyset\}: \mbox{ for each }
  \al\in (0,1], \ \langle u\rangle_\al = \bigcap_{\beta<\al} \langle u\rangle_\beta \right\}, \label{pusc}
\\
P_{USCB}(X)=\left\{  u \in K(X\times [0,1])\cup \{\emptyset\}: \mbox{ for each }
  \al\in (0,1], \ \langle u\rangle_\al = \bigcap_{\beta<\al} \langle u\rangle_\beta \right\}.\label{puscbe}
 \end{gather}
So
$ P_{USCB}(X) = \big(K(X\times [0,1])\cup \{\emptyset\}\big)\cap P_{USC}(X)$.

Clearly $P_{USC}(X)\setminus \{\emptyset\}\subset C(X\times [0,1])$
and
$P_{USCB}(X) \setminus \{\emptyset\} \subset K(X\times [0,1])$.
Choose $x\in X$. Let $u=\{(x,0), (x,1)\}$. Then $u\in K(X\times [0,1])$ but
$u\notin P_{USC}(X)$. This means that
$u\in K(X\times [0,1]) \setminus P_{USCB}(X)$
and
$u\in C(X\times [0,1]) \setminus P_{USC}(X)$.
So
  $P_{USC}(X) \setminus \{\emptyset\}\subsetneqq C(X\times [0,1])$
and $P_{USCB}(X) \setminus\{\emptyset\} \subsetneqq K(X\times [0,1])$.

It is easy to see that the conditions
(\rmn1) $X$ is compact, (\rmn2) $P_{USC}(X)  = P_{USCB}(X)$, and (\rmn3) $F_{USC}(X)=F_{USCB}(X)$, are equivalent.

Let $u\in F_{USC}(X)$.
Define $\bm{u^e}:= {\rm end}\, u $ and $\bm{u^s}:= {\rm send}\, u $.
Clearly ${\emptyset_{F(X)}}^s = \emptyset$.
Let $A$ be a subset of $F_{USC}(X)$.
Define
$\bm{A^e}:= \{u^e:  u\in A\}$ and $\bm{A^s}:= \{u^s:  u\in A\}$.
Then
\begin{equation}\label{areu}
F_{USC}(X)^e \subseteq P_{USC}(X), \ F_{USC}(X)^s \subseteq P_{USC}(X), \
F_{USCB}(X)^s \subseteq P_{USCB}(X).
\end{equation}

Let $A\in C(X)$ and $B\in K(X)$.
Then
\begin{gather*}
 \emptyset\in P_{USC}(X) \setminus F_{USC}(X)^e = \{u\in P_{USC}(X) : \langle u \rangle_0 \not= X\},
\\
A\times\{0\} \in P_{USC}(X) \setminus F_{USC}(X)^s = \{u\in P_{USC}(X) : \langle u \rangle_0 \supsetneqq \overline{\cup_{\al>0}\langle u \rangle_\al}\},
\\
B\times\{0\} \in P_{USCB}(X) \setminus F_{USCB}(X)^s = \{u\in P_{USCB}(X) : \langle u \rangle_0 \supsetneqq \overline{\cup_{\al>0}\langle u \rangle_\al}\}.
\end{gather*}
So each ``$\subseteq$'' in \eqref{areu} can be replaced by ``$\subsetneqq$''.

If $X$ is a singleton, then
$
P_{USC}(X) \setminus F_{USC}(X)^e =\{\emptyset\} $
and
$P_{USC}(X) \setminus F_{USC}(X)^s = P_{USCB}(X) \setminus F_{USCB}(X)^s  = \{X\times\{0\}\}$.

 If $X$ has at least two elements,
then each of $P_{USC}(X) \setminus F_{USC}(X)^e$,
$P_{USC}(X) \setminus F_{USC}(X)^s$ and $P_{USCB}(X) \setminus F_{USCB}(X)^s$
has elements with their heights being positive.

Assume that $X$ has at least two elements.
Choose $x,y\in X$ with $x\not=y$.
Define $u,v$ in $X\times [0,1]$ by putting
\[
\langle u\rangle_{\al}=
\left\{
   \begin{array}{ll}
    \{x\}, & \al\in (0,1],
\\
\{x,y \}, & \al=0,
   \end{array}
 \right.
\mbox{ and }
\langle v\rangle_{\al}=
    \{x\} \mbox{ for all }\al\in [0,1].
\]
Then
 $u,v\in P_{USCB}(X)$ and $S_u=S_v=1$. However,
$u\notin F_{USC}(X)^s$, $v\notin F_{USC}(X)^e$.
So
$v\in P_{USC}(X) \setminus F_{USC}(X)^e$,
$u\in P_{USC}(X) \setminus F_{USC}(X)^s$,
$u\in
P_{USCB}(X) \setminus F_{USCB}(X)^s$.

Let $u\in P_{USC}(X)$. By Proposition \ref{fus},
we can define $\bm{\overleftarrow{u}}\in F_{USC}(X)$ as follows:
for each $\al\in (0,1]$,
$
[\overleftarrow{u}]_\al= \langle u \rangle_\al
$.
 Then $\overleftarrow{u}^s \subseteq u \subseteq \overleftarrow{u}^e$.

Let $A$ be a subset of $P_{USC}(X)$.
Define
$\bm{\overleftarrow{A}}:= \{\overleftarrow{u}:  u\in A\}$.
Clearly for each $u\in F_{USC}(X)$, $\overleftarrow{u^e} = \overleftarrow{u^s} =u$, and for
each $u\in F_{USCB}(X)$, $u^s\in P_{USCB}(X)$.
So $\overleftarrow{P_{USC}(X)}= F_{USC}(X)$ and $\overleftarrow{P_{USCB}(X)}= F_{USCB}(X)$.

Clearly for each $u\in P_{USC}(X)$ and $v\in F_{USC}(X)$,
$\overleftarrow{u} = v$ if and only if $v^s\subseteq u \subseteq v^e$.
The conditions
(\rmn1) $X$ is compact, and (\rmn2) there is no $u\in P_{USC}(X)\setminus P_{USCB}(X)$ with $\overleftarrow{u}\in F_{USCB}(X)$, are equivalent.
 $X$ is compact means that $P_{USC}(X)\setminus P_{USCB}(X)=\emptyset$.
So (\rmn1)$\Rightarrow$(\rmn2).
If $X$ is not compact, pick $x\in X$, then ${\widehat{x}}^e\in P_{USC}(X)\setminus P_{USCB}(X)$ and $\overleftarrow{{\widehat{x}}^e}=\widehat{x}\in F_{USCB}(X)$.
So (\rmn2)$\Rightarrow$(\rmn1).

We can use the results in this section to discuss the properties of $F(X)$.
Below we use Corollary \ref{bpecun}(\rmn2) to show the conclusion in Remark \ref{bcp}.
Conversely, we can also use Remark \ref{bcp} to show Corollary \ref{bpecun}(\rmn2). All these two proofs are routine.

Let $u\in F(X)$. Observe that for each $\al\in (0, 1]$,
 $\langle {\rm end}\,u \rangle_\al = \langle {\rm send}\,u \rangle_\al = [u]_\al$, and  $\langle {\rm end}\,u \rangle_0 = X$ and $\langle {\rm send}\,u \rangle_0 = [u]_0$.
So
 $\langle {\rm end}\,u \rangle_\al = \bigcap_{\beta\in [0, \al)} \langle {\rm end}\,u \rangle_\beta$ and
 $\langle {\rm send}\,u \rangle_\al = \bigcap_{\beta\in [0, \al)} \langle {\rm send}\,u \rangle_\beta$.
As
$\langle {\rm end}\,u \rangle_0 \in C(X)$ and $\langle {\rm send}\,u \rangle_0 = [u]_0\in C(X)$, by Corollary \ref{bpecun}(\rmn2),
the conditions
(\rmn1) ${\rm end}\,u$ is closed in $(X\times [0,1],   \overline{d})$,
(\rmn2) ${\rm send}\,u$ is closed in $(X\times [0,1],   \overline{d})$,
and
(\rmn3) $[u]_\al \in C(X) \cup \{\emptyset\}$ for all $\al\in (0,1]$, are equivalent.
Clearly $u\in F_{USC}(X)$ is equivalent to (\rmn3).
So we show the conclusion in Remark \ref{bcp}.

\section{Characterization of compactness in $(F_{USCG} (X), H_{\rm end})$}
\label{cmg}

In this section, we give the characterizations of
relatively compact sets,
totally bounded sets, and compact sets in
$(F_{USCG} (X), H_{\rm end})$, respectively.
We point out that
these results improve
the
characterizations of
relatively compact sets,
totally bounded sets, and compact sets in
$(F_{USCG} (\mathbb{R}^m), H_{\rm end})$ given in our previous work \cite{huang}, respectively.

\begin{itemize}
 \item A subset $Y$ of a topological space $Z$ is said to be \emph{compact} if for every set $I$
and every family of open sets, $O_i$, $i\in I$, such that $Y\subset \bigcup_{i\in I} O_i$ there exists
a finite family $O_{i_1}$, $O_{i_2}$ \ldots, $O_{i_n}$ such that
$Y\subseteq O_{i_1}\cup O_{i_2}\cup\ldots \cup O_{i_n}$.
In
the case of a metric topology, the criterion for compactness becomes that any sequence in $Y$ has a subsequence convergent in $Y$.

 \item
A \emph{relatively compact} subset $Y$ of a topological space $Z$ is a subset with compact closure. In the case of a metric topology, the criterion for relative compactness becomes that any sequence in $Y$ has a subsequence convergent in $X$.

 \item Let $(X, d)$ be a metric space. A set $U$ in $X$ is \emph{totally bounded} if and only if, for each $\varepsilon>0$, it contains a finite $\varepsilon$ approximation, where an $\varepsilon$ approximation to $U$ is a subset $S$ of $U$ such that $d(x,S)<\varepsilon$ for each $x\in U$.
An $\varepsilon$ approximation to $U$ is also called an $\varepsilon$-net of $U$.

It is known that $U$ in $X$ is totally bounded if and only if, for each $\varepsilon>0$, there is a finite \emph{weak}
$\varepsilon$\emph{-net} of $U$, where
a \emph{weak}
$\varepsilon$\emph{-net} of $U$ is a subset $S$ of $X$
satisfying that $d(x,S)<\varepsilon$ for each $x\in U$.

\end{itemize}

Let $(X, d)$ be a metric space. A set $U$ is compact in $(X,d)$ implies that $U$ is relatively compact
in $(X,d)$, which in turn
implies that $U$ is totally bounded in $(X,d)$.
Let $Y$ be a subset of $X$ and $A$ a subset of $Y$.
Then
$A$ is totally bounded in $(Y,d)$ if and only if
$A$ is totally bounded in $(X,d)$.
See also \cite{huang719} or
the papers: 
H. Huang, Properties of several fuzzy set spaces (v1-v13), chinaXiv:202107.00011, some of them appear earlier than \cite{huang719}.
The paper:
H. Huang, Characterizations of several convergence structures on fuzzy sets, arXiv:1910.02205v1, submitted on 5 Oct 2019, is an earlier version of chinaXiv:202107.00011.

\begin{tm} \cite{huang719}\label{tbe}
  Let $(X,d)$ be a metric space and $\mathcal{D}\subseteq K(X)$.
   Then $\mathcal{D}$ is totally bounded in $(K(X), H)$
if and only if
$\mathbf{ D} =   \bigcup \{C:  C \in  \mathcal{D} \} $ is totally bounded in $(X,d)$.
\end{tm}

\begin{tm} \cite{greco} \label{rce}
  Let $(X,d)$ be a metric space and $\mathcal{D}\subseteq K(X)$. Then $\mathcal{D} $ is relatively compact in $(K(X), H)$
if and only if
$\mathbf{ D} =   \bigcup \{C:  C \in  \mathcal{D} \} $ is relatively compact in $(X,d)$.
\end{tm}

\begin{tm} \cite{huang719} \label{come}
 Let $(X,d)$ be a metric space and $\mathcal{D}\subseteq K(X)$. Then
    the following are equivalent:
    \\
    (\romannumeral1) \ $\mathcal{D} $ is compact in $(K(X), H)$;
        \\
 (\romannumeral2) \
$\mathbf{ D} =   \bigcup \{C:  C \in  \mathcal{D} \} $ is relatively compact in $(X, d)$
and
$\mathcal{D} $ is closed in $(K(X), H)$;
    \\
 (\romannumeral3) \
$\mathbf{ D} =   \bigcup \{C:  C \in  \mathcal{D} \} $ is compact in $(X, d)$
and
$\mathcal{D} $ is closed in $(K(X), H)$.
\end{tm}

Let $u\in F_{USC}(X)$.
Then $F_{USC}(X)^e\subseteq C(X\times [0,1])$
(see the paragraph below Remark \ref{bcp}).
Define
$g: (F_{USC}(X), H_{\rm end}) \to (C(X\times [0,1]), H)$ by
$g(u) = {\rm end}\, u$, for every $u\in F_{USC}(X)$.
Then $g(F_{USC}(X))=F_{USC}(X)^e$.
Clearly for each $u,v\in F_{USC}(X)$, $H_{\rm end} (u,v)=H(g(u), g(v))$.
This means that $g$ is an isometric embedding
of
 $(F_{USC}(X), H_{\rm end}) $ in $(C(X\times [0,1]), H)$.

The following representation theorem for $F_{USC}(X)^e$ and $F_{USCG}(X)^e$ follows immediately
from Proposition \ref{fus}.

\begin{pp} \label{repu}
Let $U$ be a subset of $X\times [0,1]$.
Then $U\in F_{USC}(X)^e$ (respectively, $U\in F_{USCG}(X)^e$)
if and only if the following properties (\romannumeral1)-(\romannumeral3) are true.
\\
(\romannumeral1) For each $\alpha \in (0, 1]$, $\langle U \rangle_\al \in C(X) \cup \{\emptyset\}$ (respectively, $\langle U \rangle_\al \in K(X) \cup \{\emptyset\}$).
\\
(\romannumeral2) \
For each $\alpha \in (0, 1]$,
$\langle U\rangle_\al = \bigcap_{\beta<\al} \langle U\rangle_\beta $.
\\
(\romannumeral3) \ $\langle U\rangle_0 = X $.
\end{pp}

\begin{pp}
  \label{ucen}
Let $c_0, c_1$ in $[0,1]$ with $c_0\leq c_1$ and
let $U\in C(X\times [c_0,c_1])\cup\{\emptyset\}$. Then the following (\romannumeral1) and (\romannumeral2)
are equivalent.
\\
(\romannumeral1) For each $\alpha$ with $c_0 < \al \leq c_1$,
$\langle U\rangle_\al = \bigcap_{\beta\in [c_0, \al)} \langle U\rangle_\beta $.
\\
(\romannumeral2)
For each $\alpha,\beta$ with $c_0\leq \beta < \al \leq c_1$,
$\langle U\rangle_\al \subseteq \langle U\rangle_\beta $.
\end{pp}

\begin{proof} The proof is routine.
  (\romannumeral1)$\Rightarrow$(\romannumeral2) is obviously.

  Suppose that (\romannumeral2) is true.
 To show that (\romannumeral1) is true, let $\al\in (c_0, c_1]$.
      From (\romannumeral2), $\langle U\rangle_\al \subseteq \bigcap_{\beta\in [c_0, \al)} \langle U\rangle_\beta $.
So we only need to prove
that
$\langle U\rangle_\al \supseteq \bigcap_{\beta\in [c_0, \al)} \langle U\rangle_\beta $.
   To do this, let
   $x\in \bigcap_{\beta\in [c_0, \al)} \langle U\rangle_\beta$.
   This means that for each $\beta\in [c_0,\al)$, $(x,\beta) \in U$.
  Observe that $\lim_{\beta\to \al-} \overline{d}((x,\beta), (x, \al)) = 0$, so $(x, \al)\in U$, by the closedness of
$U$.
Hence $x\in \langle U \rangle_\al$.
Since $x \in \bigcap_{\beta\in [c_0, \al)} \langle U\rangle_\beta$ is arbitrary,
we have that $\langle U\rangle_\al \supseteq \bigcap_{\beta\in [c_0, \al)} \langle U\rangle_\beta $.
   So (\rmn1) is true. Thus
 (\romannumeral2)$\Rightarrow$(\romannumeral1) is proved.

\end{proof}

\begin{pp} \label{repc}
Let $U\in C(X\times [0,1])$.
Then
$U\in F_{USC}(X)^e$
if and only if $U$ has the following properties:
\\
(\romannumeral1) \
for each $\alpha,\beta$ with $0\leq \beta < \al \leq 1$,
$\langle U\rangle_\al \subseteq \langle U\rangle_\beta $, and
\\
(\romannumeral2) \ $\langle U\rangle_0 = X$.
\end{pp}

\begin{proof}
Since
 $U\in C(X\times [0,1])$, then clearly $\langle U \rangle_\al \in C(X) \cup \{ \emptyset \}$
for all $\al\in [0,1]$.
Thus the desired result follows immediately from Propositions \ref{repu} and \ref{ucen}.

\end{proof}

As a shorthand, we denote the sequence $x_1, x_2, \ldots, x_n, \ldots$ by $\{x_n\}$.

\begin{pp}\label{usce}
$F_{USC}(X)^e$ is a closed subset of $(C(X \times [0,1]), H)$.
\end{pp}

\begin{proof}
  Let $\{u_n^e:   n=1,2,\ldots \}$ be a sequence in $F_{USC}(X)^e$
with
$\{u_n^e\}$ converging to $U$ in $(C(X \times [0,1]), H)$.
To show the desired result, we
only need to show
that $U \in F_{USC}(X)^e$.

We claim that
\\
(\romannumeral1) \
for each $\alpha,\beta$ with $0\leq \beta < \alpha \leq 1$,
$\langle U\rangle_\al \subseteq \langle U\rangle_\beta $;
\\
(\romannumeral2) \ $\langle U\rangle_0 = X $.

To show (\romannumeral1),
let
 $\alpha,\beta$ in $[0,1]$ with $\beta < \alpha$,
and let
 $x \in \langle U\rangle_\al$, i.e. $(x,\al) \in U$.
By Theorem \ref{hkg},
$\lim_{n\to \infty}^{(K)} u_n^e \, = U$.
Then there is a sequence $\{(x_n, \al_n)\}$ satisfying $(x_n, \al_n) \in u_n^e$ for $n=1,2,\ldots$
and
$\lim_{n\to\infty} \overline{d} ((x_n, \al_n), (x,\al)) = 0$.
Hence there is an $N$ such that $\al_n > \beta$ for all $n\geq N$.
Thus $(x_n, \beta) \in u_n^e$ for all $n\geq N$. Note that
$\lim_{n\to\infty} \overline{d} ((x_n, \beta), (x,\beta)) = 0$.
Then $(x,\beta) \in \lim_{n\to \infty}^{(K)} u_n^e = U$.
This means that
$x\in \langle U\rangle_\beta$.
So (\romannumeral1) is true.

Clearly $\langle U \rangle_0 \subseteq X$.
From $\lim_{n\to \infty}^{(K)} u_n^e \, = U$
and $\langle u_n^e \rangle_0 = X$,
we have that
$\langle U \rangle_0 \supseteq X$.
Thus
$\langle U \rangle_0 = X$.
So (\romannumeral2) is true.

By Proposition \ref{repc}, (\romannumeral1) and (\romannumeral2) imply that
 $U \in F_{USC}(X)^e$.

\end{proof}

\begin{re}
  {\rm
(\rmn1) By \eqref{pusc}, \eqref{puscbe} and
  Proposition \ref{ucen},
\\
$P_{USC}(X)=\left\{  u \in C(X\times [0,1])\cup \{\emptyset\}: \mbox{ for each } \alpha,\beta \mbox{ with } 0\leq \beta < \alpha \leq 1, \langle u\rangle_\al \subseteq \langle u\rangle_\beta \right\}$.
\\
$P_{USCB}(X)=\left\{  u \in K(X\times [0,1])\cup \{\emptyset\}: \mbox{ for each } \alpha,\beta \mbox{ with } 0\leq \beta < \alpha \leq 1, \langle u\rangle_\al \subseteq \langle u\rangle_\beta \right\}$.

(\rmn2) $P_{USC}(X)$ is a closed subset of $(C(X \times [0,1])\cup\{\emptyset\}, H)$.

Let $\{u_n:   n=1,2,\ldots \}$ be a sequence in $P_{USC}(X)$
with
$\{u_n\}$ converging to $u$ in $(C(X \times [0,1])\cup\{\emptyset\}, H)$.
To show (\rmn2), we
only need to show
that $u \in P_{USC}(X)$.
By (\rmn1), to do this,
it suffices to
show that (a)
for each $\alpha,\beta$ with $0\leq \beta < \alpha \leq 1$,
$\langle u\rangle_\al \subseteq \langle u\rangle_\beta $.
The proof of (a) is similar to that of (\rmn1) in
the proof of
Proposition \ref{usce}.
Thus $u \in P_{USC}(X)$. So (\rmn1) is proved.

(\rmn2)
Let $a\in [0,1]$.
Proposition \ref{ace}(\rmn2) says that
$F^{'a}_{USC}(X)$ is a closed subset of $(F_{USC}(X), H_{\rm end})$.
Then by Proposition \ref{usce}, we have that $F^{'a}_{USC}(X)^e$ is a closed subset of $(C(X \times [0,1]), H)$.

}
\end{re}

We use $(\widetilde{X}, \widetilde{d})$ to denote the completion of $(X, d)$.
We see $(X, d)$ as a subspace of $(\widetilde{X}, \widetilde{d})$.

If there is no confusion,
 we also
use $H$ to denote the Hausdorff metric
on
 $C(\widetilde{X})$ induced by $\widetilde{d}$.
  We also
use $H$ to denote the Hausdorff metric
 on $C(\widetilde{X}\times [0,1])$ induced by $\overline{\widetilde{d}}$.
 We
 also use $H_{\rm end}$ to denote the endograph metric on $F_{USC}(\widetilde{X})$
given by using $H$ on
$C(\widetilde{X} \times [0,1])$.

$F(X)$ can be naturally embedded into $F(\widetilde{X})$. An embedding $\bm{j}$
from $F(X)$ to $F(\widetilde{X})$ is defined as follows.

Let $u\in F(X)$. We can define $j(u)\in F(\widetilde{X})$ as
\[
j(u) (t)
=
\left\{
  \begin{array}{ll}
    u(t), & t\in X,\\
0, & t\in \widetilde{X} \setminus X.
  \end{array}
\right.
\]

Let $U\subseteq X$. If $U$ is compact in $(X,d)$, then $U$ is compact in $(\widetilde{X}, \widetilde{d})$.
So
if $u\in F_{USCG}(X)$, then $j(u)\in F_{USCG} (\widetilde{X})$
because $[j(u)]_\al = [u]_\al \subseteq K(\widetilde{X})\cup \{\emptyset\}$ for each $\al\in (0,1]$.
Also, for $u,v\in F_{USCG}(X)$, $H_{\rm end}(u, v) = H_{\rm end} (j(u), j(v))$.

Define
$k: F_{USCG}(X) \to F_{USCG}(\widetilde{X})$ as
$k(u) = j(u)$ for each $u\in F_{USCG}(X)$.
Then
$k$ is an isometric embedding of $(F_{USCG}(X), H_{\rm end})$ in $(F_{USCG}(\widetilde{X}), H_{\rm end})$.

In the sequel, if there is no confusion, we treat
$(F_{USCG}(X), H_{\rm end})$ as a subspace of $(F_{USCG}(\widetilde{X}), H_{\rm end})$ by
identifying $u$ in $F_{USCG}(X)$ with $j(u)$ in $F_{USCG}(\widetilde{X})$,
and so we treat a subset $U$ of $F_{USCG} (X)$ as a subset of $F_{USCG}(\widetilde{X})$.

Let $u\in F_{USCG}(X)$.
Clearly
  for each $\al\in (0,1]$,
$\langle j(u)^e\rangle_\al= \langle u^e\rangle_\al$.

Suppose that
$U$ is a subset of $F_{USC} (X)$ and $\al\in [0,1]$.
For
writing convenience,
we denote
\begin{itemize}
  \item  $U(\al):= \bigcup_{u\in U} [u]_\al$, and

\item  $U_\al : =  \{[u]_\al: u \in U\}$.
\end{itemize}
Here we mention that an empty union is $\emptyset$.

\begin{pp}\label{aeng}
 Let $u,v\in F_{USC}(X)$ and
$\varepsilon\in (0,1]$. Let $\alpha,\beta\in [0,1]$ with $\alpha-\beta\geq \varepsilon$. Set $a:= H^*([u]_{\al}, \overline{\{v>\beta\}})$
and $b:= \sup\{\overline{d}((y,\al), {\rm end}\, v): y\in [u]_\al\}$.
\\
(\rmn1) If $b<\varepsilon$, then $a\leq b$.
\\
(\rmn2)
If
$
H^*({\rm end}\, u, {\rm end}\, v)<\varepsilon$, then $H^*([u]_{\alpha}, [v]_\beta)\leq H^*({\rm end}\, u, {\rm end}\, v)<\varepsilon$.
\end{pp}

\begin{proof}

First we show (\rmn1).
Assume that $b<\varepsilon$. If $[u]_\al=\emptyset$, then $a=0\leq b$.

Suppose that $[u]_\al \not=\emptyset$.
Let $y\in [u]_\al$.
Set
\begin{gather*}
  \xi:=\inf\{\overline{d}((y,\al), (z,\gamma)): (z,\gamma)\in {\rm end}\, v \mbox{ with } \gamma\leq \beta  \},\\
\eta:=\left\{
        \begin{array}{ll}
                  \inf\{\overline{d}((y,\al), (z,\gamma)): (z,\gamma)\in {\rm end}\, v \mbox{ with } \gamma>\beta  \}, & \{v>\beta\}\not=\emptyset, \\
+\infty, & \{v>\beta\} = \emptyset.
\end{array}
      \right.
\end{gather*}
 Then
$
\overline{d}((y,\al), {\rm end}\, v)=\inf\{\overline{d}((y,\al), (z,\gamma)): (z,\gamma)\in {\rm end}\, v \}=\xi
\wedge
\eta
$. This implies that $
\overline{d}((y,\al), {\rm end}\, v)=\xi$ or $\overline{d}((y,\al), {\rm end}\, v)=\eta$.

Notice that for each $(z,\gamma) \in {\rm end}\, v$ with $\gamma \leq \beta$,
$\overline{d}((y,\al), (z,\gamma)) \geq |\al-\gamma| \geq |\al-\beta|\geq \varepsilon$.
Thus $\xi\geq \varepsilon$.
Since
$ \overline{d}((y,\al), {\rm end}\, v) \leq b <\varepsilon$,
it follows that
 $\{v>\beta\}\not=\emptyset$ and
 \begin{equation}\label{denr}
\overline{d}((y,\al), {\rm end}\, v)=\eta=\inf\{\overline{d}((y,\al), (z,\gamma)): (z,\gamma)\in {\rm end}\, v \mbox{ with } \gamma> \beta\}.
\end{equation}
As
for each
$(z,\gamma)\in {\rm end}\, v \mbox{ with } \gamma > \beta$,
$\overline{d}((y, \al), (z, \gamma)) \geq d(y, z) \geq d(y, \{v>\beta\}) = d(y, \overline{\{v>\beta\}})$,
it follows from \eqref{denr} that
\begin{equation*}
 \overline{d}((y,\al), {\rm end}\, v)
\geq d(y, \overline{\{v>\beta\}}).
\end{equation*}
Thus
$b = \sup\{\overline{d}((y,\al), {\rm end}\, v): y\in [u]_\al\} \geq
 \sup\{d(y, \overline{\{v>\beta\}}: y\in [u]_\al \} = a$.
So (\rmn1) is true.

Now we show (\rmn2).
Clearly $H^*([u]_{\alpha}, [v]_\beta)\leq a$ and $b \leq H^*({\rm end}\, u, {\rm end}\, v)$.
If $
H^*({\rm end}\, u, {\rm end}\, v)<\varepsilon$, then
$b<\varepsilon$, and hence by (\rmn1), $H^*([u]_{\alpha}, [v]_\beta)\leq a\leq
b \leq H^*({\rm end}\, u, {\rm end}\, v)<\varepsilon$.
So (\rmn2) is true.

\end{proof}

Here we mention that Proposition \ref{aeng} and its proof remain true if $\overline{d}$ is replaced by $d'$ and
$H^*$ on $C(X\times [0,1])\cup\{\emptyset\}$ is replaced
 by
$H'^*$ on $C(X\times [0,1])\cup\{\emptyset\}$.

\begin{lm}
   \label{tbfegnum}
  Let $U$ be a subset of $F_{USCG} (X)$. If $U$ is totally bounded in $(F_{USCG} (X), H_{\rm end})$,
then
$U(\al)$
is totally bounded in $(X,d)$ for each $\al \in (0,1]$.
\end{lm}

\begin{proof} The proof is similar to that of the necessity part of Theorem 7.8 in \cite{huang719}. See also the papers: H. Huang, Properties of several fuzzy set spaces (v1-v13), chinaXiv:202107.00011.

Let $\al \in (0,1]$.
To show
that
$U(\al)$
is totally bounded in $X$, we only need to show that
each sequence in $U(\al)$
has a Cauchy subsequence.

Let $\{x_n\}$ be a sequence in $U(\al)$. Then there is a sequence $\{u_n\}$ in $U$ with $x_n \in  [u_n]_\al$ for $n=1,2,\ldots$.
Since $U$ is totally bounded in $(F_{USCG} (X), H_{\rm end})$,
$\{u_n\}$ has
 a Cauchy subsequence $\{u_{n_l}\}$ in $(F_{USCG} (X), H_{\rm end})$. So given $\varepsilon\in (0, \al)$, there is a $K( \varepsilon) \in \mathbb{N}$
such that
$
H_{\rm end} (u_{n_l},  u_{n_K})  <   \varepsilon
$
for all $l \geq K$.
Thus by Proposition \ref{aeng}(\rmn2),
\begin{equation}\label{cutsgm}
H^*([u_{n_l}]_\al,       [u_{n_K}]_{\al-\varepsilon}) <  \varepsilon
\end{equation}
for all $l \geq K$.
We claim that
(a) $\bigcup_{l=1} ^{+\infty} [u_{n_l}]_\al $ is totally bounded in $(X,d)$.

To show (a), it suffices to show that
for each $\lambda>0$, there exists a finite weak $\lambda$-net of  $ \bigcup_{l=1} ^{+\infty} [u_{n_l}]_\al$.

Let $\lambda>0$. Set $\varepsilon= \min\{\lambda/2, \al/2\} $.
Then $\varepsilon \in (0, \al)$. Hence there is a $K(\varepsilon)$ such that \eqref{cutsgm} holds for all $l\geq K$.
Set $A:=\bigcup_{l=1} ^{K-1}[u_{n_l}]_\al  \cup [u_{n_K}]_{\al-\varepsilon}$.
Since $A$ is compact,
there is a finite $\varepsilon$ approximation $\{z_j\}_{j=1}^{m}$ to $A$. Below we show that $\{z_j\}_{j=1}^{m}$
is a finite weak $\lambda$-net of $\bigcup_{l=1} ^{+\infty} [u_{n_l}]_\al$.

Let $z \in \bigcup_{l=1} ^{+\infty} [u_{n_l}]_\al$.
If $z \in \bigcup_{l=1} ^{K-1} [u_{n_l}]_\al$, then clearly $d(z, \{z_j\}_{j=1}^{m}) < \varepsilon<\lambda$.
If $z \in \bigcup_{l=K} ^{+\infty} [u_{n_l}]_\al$,
then
by \eqref{cutsgm},
  there exists a $y_z \in [u_{n_K}]_{\al-\varepsilon}$
such that
$d(z, y_z) <\varepsilon $,
and hence
$d(z, \{z_j\}_{j=1}^{m}) \leq  d(z, y_z) + d(y_z, \{z_j\}_{j=1}^{m} ) < 2\varepsilon \leq \lambda$.
Thus for each $z \in \bigcup_{l=1} ^{+\infty} [u_{n_l}]_\al$,
$d(z, \{z_j\}_{j=1}^{m}) < \lambda$.
Hence
$\{z_j\}_{j=1}^{m}$
is a finite weak $\lambda$-net of $\bigcup_{l=1} ^{+\infty} [u_{n_l}]_\al$.
So (a) is proved.

As $\{x_{n_l}\}$ is a sequence in $\bigcup_{l=1} ^{+\infty} [u_{n_l}]_\al$,
 (a) implies that
$\{x_{n_l}\}$
has a Cauchy subsequence.
 Then $\{x_n\}$ also has a Cauchy subsequence
because $\{x_{n_l}\}$ is a subsequence of $\{x_n\}$.
So $U(\al)$
is totally bounded in $X$. This completes the proof.

\end{proof}

\begin{re}{\rm
 It is easy to see that for a totally bounded set $U$ in $(F_{USCG} (X), H_{\rm end})$ and $\al\in (0,1]$,
$U(\al) = \emptyset$ is possible even if $U \not= \emptyset$.
}
\end{re}

Let $u\in F(X)$ and $0\leq r \leq t\leq 1$.
We use the symbol $\bm{{\rm end}_r^t\, u }$ to denote
the subset of ${\rm end}\, u $ given by
$$
{\rm end}_r^t\, u  :=    {\rm end}\, u \cap  ([u]_r \times [r, t] ).$$
For simplicity, we write
${\rm end}_r^1\, u  $ as ${\rm end}_r\, u  $.
We can see that ${\rm end}_0\, u  = {\rm send}\, u$.

\begin{tm}
 \label{rcfegnum}
  Let $U$ be a subset of $F_{USCG} (X)$. Then $U$ is relatively compact in $(F_{USCG} (X), H_{\rm end})$
if and only if
$U(\al)$
is relatively compact in $(X, d)$ for each $\al \in (0,1]$.

\end{tm}

\begin{proof}

\textbf{\emph{Necessity}}. Suppose that $U$ is relatively compact in $(F_{USCG} (X), H_{\rm end})$.
Let $\al\in (0,1]$.
Then by Lemma \ref{tbfegnum},
 $U(\al)$ is totally bounded.
Hence
$U(\al)$ is relatively compact in $(\widetilde{X}, \widetilde{d})$.

To show that $U(\al)$ is relatively compact in $(X,d)$, we proceed by  contradiction.
If this were not the case, then
there exists a sequence $\{x_n\}$ in $U(\al)$
such that $\{x_n\}$ converges to $x\in \widetilde{X} \setminus X$ in $(\widetilde{X}, \widetilde{d})$.

Assume that $x_n \in [u_n]_\al$ and $u_n \in U$, $n=1,2,\ldots$.
From the relative compactness of $U$,
there is a subsequence $\{u_{n_k}, k=1,2,\ldots\}$ of $\{u_n\}$
such that 
$\{u_{n_k}\}$
converges to $u$ in $(F_{USCG}(X), H_{\rm end})$, which is equivalent to
 $\{j(u_{n_k})\}$ converges to $j(u)$ in $(F_{USCG}(\widetilde{X}), H_{\rm end})$.
Then by Remark \ref{hmr}, $\lim_{k\to\infty}^{(\Gamma)} j(u_{n_k}) \, = j(u)$;
that is,
$\lim_{k\to \infty}^{(K)} j(u_{n_k})^e \, = j(u)^e$ according to $(\widetilde{X}\times[0,1], \overline{\widetilde{d}})$.
Notice
that
$(x_{n_k}, \al) \in j(u_{n_k})^e $ for $k=1,2,\ldots$, and that
$\{(x_{n_k}, \al)\}$ converges to $(x, \al)$ in $(\widetilde{X}\times[0,1], \overline{\widetilde{d}})$.
Thus
$(x,\al) \in j(u)^e$. This means that $x\in \langle j(u)^e\rangle_\al =\langle u^e\rangle_\al$, which contradicts $x\in \widetilde{X} \setminus X$.

It can be seen that
the necessity part of Theorem 7.9 in \cite{huang719} can be verified in a similar fashion to the necessity
part of this theorem.

\textbf{\emph{Sufficiency}}. \
Suppose that
$U(\al)$
is relatively compact in $(X, d)$ for each $\al \in (0,1]$.
To show that $U$ is relatively compact in $(F_{USCG} (X), H_{\rm end})$,
we only need to show
that each sequence in $U$ has a convergent subsequence in $(F_{USCG} (X), H_{\rm end})$.

Let $\{u_n\}$ be a sequence in $U$. We split the proof into two cases.

\textbf{Case (\rmn1)}
$\liminf_{n\to\infty} S_{u_n} = 0$; that is, there is a subsequence $\{u_{n_k}\}$ of $\{u_n\}$
such that
$\lim_{k\to\infty} S_{u_{n_k}} = 0$. Then clearly $H_{\rm end} (u_{n_k}, \emptyset_{F(X)}) = S_{u_{n_k}} \to 0$
as $k \to \infty$.
Since $\emptyset_{F(X)} \in F_{USCG} (X)$, $\{u_{n_k}\}$ is a convergent subsequence in $(F_{USCG} (X), H_{\rm end})$.

\textbf{Case (\rmn2)}
$\liminf_{n\to\infty} S_{u_n} >0$; that is, there is a $\xi>0$ and an $N \in \mathbb{N}$
such that $[u_n]_{\xi} \not= \emptyset$
for all $n\geq N$.

First we claim the following property:
\begin{description}
  \item[(a)]
Let $\al\in (0,1]$ and let $S$ be a subset of $U$ with $[u]_\al \not= \emptyset$ for each $u\in S$. Then
$\{ {\rm end}_\al u: u\in S \}$ is a relatively compact set in $(K(X\times [\al,1]), H)$.
\end{description}

Below we show (a).
First we verify that for each $u\in S$, ${\rm end}_\al\, u \in K(X\times [\al,1])$. Let $u\in S$. Clearly ${\rm end}_\al\, u \not= \emptyset$ as
${\rm end}_\al\, u\supseteq [u]_\al\times \{\al\}$.
Note that for each $\beta\in [\al,1]$,
$\langle {\rm end}_\al\, u  \rangle_\beta = [u]_\beta\in K(X)\cup\{\emptyset\}$.
So
by Proposition \ref{bpec}(\rmn3),
${\rm end}_\al\, u \in K(X\times [\al,1])$.

As $U(\al)$ is relatively compact in $(X,d)$,
$U(\al) \times [\al, 1]$ is relatively compact in $(X\times [\al,1], \overline{d})$.
Since
$\bigcup_{u\in S} {\rm end}_\al u $ is
 a subset of $ U(\al) \times [\al, 1]$,
then
$\bigcup_{u\in S} {\rm end}_\al u $
is also a
relatively compact set in $(X\times [\al,1], \overline{d})$.
Thus
by Theorem \ref{rce}, $\{ {\rm end}_\al u: u\in S \}$ is relatively compact in $(K(X\times [\al,1]), H)$.
So affirmation (a) is true.

 Take a sequence $ \{\alpha_k, \ k=1,2,\ldots\}$ which satisfies that $0<\al_{k+1} <  \al_k \leq \min\{\xi, \frac{1}{k}\}$ for $k=1,2,\ldots$. We can see that $\alpha_k \to 0$ as $k\to\infty$.

By affirmation (a),
$\{ {\rm end}_{\al_1}\, u_n: n=N,N+1,\ldots\}$ is relatively compact in $(K(X\times [\al_1,1]), H)$.
So there is a subsequence $\{u_{n}^{(1)}\}$ of $\{u_n: n\geq N\}$ and $v^1 \in K(X\times [\al_1,1])$ such that $H({\rm end}_{\al_1}\, u_{n}^{(1)},  v^1 ) \to 0$.
Clearly, $\{u_{n}^{(1)}\}$ is also a subsequence of $\{u_n\}$.

Again using affirmation (a),
$\{ {\rm end}_{\al_2}\, u_n^{(1)} \}$ is relatively compact in $(K(X\times [\al_2,1]), H)$.
So there is a subsequence $\{u_{n}^{(2)}\}$ of $\{u_n^{(1)}\}$ and $v^2 \in K(X\times [\al_2,1])$ such that $H({\rm end}_{\al_2}\, u_{n}^{(2)},  v^2 ) \to 0$ as $n\to\infty$.

Repeating the above procedure, we can obtain $\{u_{n}^{(k)}\}$
and
$v^{k}\in K(X\times [\al_k, 1])$, $k=1,2,\ldots$,
such that
for each $k=1,2,\ldots$, $\{u_{n}^{(k+1)}\}$ is a subsequence of $\{u_{n}^{(k)}\}$
and
$H({\rm end}_{\al_k}\, u_{n}^{(k)},  v^k ) \to 0$ as $n\to\infty$.

We can see that
the sequence $\{u_n^{(n)}:n=1,2,\ldots\}$ satisfies that
for each $\al_k$, $k=1,2,\ldots$,
$\lim_{n\to\infty} H({\rm end}_{\al_k}\, u_n^{(n)},   v^{k}) = 0$.
Since for each $k\in \mathbb{N}$, $\{u_{n}^{(k)}\}$ is a subsequence of $\{u_n: n\geq N\}$, we have that
for each $n,k\in \mathbb{N}$, $[u_{n}^{(k)}]_\xi \not= \emptyset$.

We claim that
\begin{description}
\item[(b)] Let $k_1$ and $k_2$ be in $\mathbb{N}$
with $k_1 < k_2$. Let $k\in \mathbb{N}$. Then
\\
(b-\romannumeral1) \ $
\langle v^{k_1} \rangle_{\al_{k_1}} \subseteq \langle v^{k_2} \rangle_{\al_{k_1}}$,
\\
(b-\romannumeral2)
 $\langle v^{k_1} \rangle_\al = \langle v^{k_2} \rangle_\al$ when $\al\in (\al_{k_1}, 1]$,
 \\
(b-\romannumeral3) $v^k \subseteq v^{k+1}$,
\\
(b-\rmn4) for each $\alpha,\beta$ in $[\al_k,1]$ with $\beta < \alpha$,
$\langle v^k\rangle_\al \subseteq \langle v^k\rangle_\beta $,
\\
(b-\rmn5) for each $\al \in (\al_k, 1]$,
$\langle v^k\rangle_\al = \bigcap_{\beta\in [\al_k, \al)} \langle v^k\rangle_\beta $,
\\
(b-\rmn6) for each $\al\in [0,1]$,
$\langle v^k \rangle_\alpha \in K(X) \cup \{\emptyset\}$.
 \end{description}

Below we show (b).
Note that $\{ u_n^{(k_2)}\}$ is a subsequence of $\{ u_n^{(k_1)}\}$
and that
$\al_{k_2} < \al_{k_1}$. Thus
by Theorem \ref{hkg}, for each $\al\in [\al_{k_1}, 1]$,
 \begin{align} \label{cner}
    &             \langle v^{k_1} \rangle_{\al}  \times \{\al\}   \nonumber \\
&=  \liminf_{n\to \infty} {\rm end}_{\al_{k_1}}\, u_n^{(k_1)} \cap (X \times \{\al\}) \nonumber\\
&
\subseteq \liminf_{n\to \infty} {\rm end}_{\al_{k_2}}\, u_n^{(k_2)} \cap (X \times \{\al\}) \nonumber\\
&
=\langle v^{k_2} \rangle_{\al}  \times \{\al\}.
                       \end{align}
So (b-\romannumeral1) is true.

Let $\al\in [0,1]$ with $\al > \al_{k_1}$.
Set $A :=  \limsup_{n\to \infty} {\rm end}_{\al_{k_1}}\, u_n^{(k_2)}  \cap (X \times \{\al\})$
and
$B
:=\limsup_{n\to \infty} {\rm end}_{\al_{k_2}}\, u_n^{(k_2)}  \cap (X \times \{\al\})$. We show that $A=B$.

Clearly $A\subseteq B$.
Given
$(x,\al) \in B$.
Then there is a sequence
$\{(x_m, \beta_m)\}$ and a subsequence
$\{u_{n_m}^{(k_2)}\}$ of $\{u_n^{(k_2)}\}$
such that for each $m=1,2,\ldots$, $(x_m, \beta_m) \in
 {\rm end}_{\al_{k_2}}\, u_{n_m}^{(k_2)}$,
and
$\overline{d} ((x_m, \beta_m), (x,\al))\to 0$
as $m\to\infty$.
So
there is an $M$ such that for all $m>M$, $\beta_m >  \al_{k_1}$, i.e.
$(x_m, \beta_m) \in
{\rm end}_{\al_{k_1}}\, u_{n_m}^{(k_2)}
$.
Thus
$(x, \al) \in
\limsup_{n\to \infty} {\rm end}_{\al_{k_1}}\, u_n^{(k_2)}
$.
Clearly $(x, \al) \in X\times \{\al\}$.
Hence $(x, \al) \in A$.
Since $(x,\al) \in B$ is arbitrary,
it follows
that
$
B\subseteq A$.
So
$A=B$.

Now by Theorem \ref{hkg}, for each $\al\in (\al_{k_1}, 1]$,
 \begin{align}  \label{cme}
    &             \langle v^{k_1} \rangle_\al  \times \{\al\} \nonumber \\
&=  \limsup_{n\to \infty} {\rm end}_{\al_{k_1}}\, u_n^{(k_1)}  \cap (X \times \{\al\})\nonumber \\
&\supseteq\limsup_{n\to \infty} {\rm end}_{\al_{k_1}}\, u_n^{(k_2)}  \cap (X \times \{\al\})\nonumber\\
&
=\limsup_{n\to \infty} {\rm end}_{\al_{k_2}}\, u_n^{(k_2)}  \cap (X \times \{\al\}) \nonumber\\
&
=\langle v^{k_2} \rangle_\al  \times \{\al\}.
                       \end{align}
Hence
 by \eqref{cner} and \eqref{cme},
$ \langle v^{k_1} \rangle_\al =  \langle v^{k_2} \rangle_\al $ for $\al\in (\al_{k_1}, 1]$. So (b-\romannumeral2) is true.

(b-\romannumeral3)
follows immediately from (b-\romannumeral1) and (b-\romannumeral2).
The proof of (b-\rmn4) is similar to that of
(\rmn1) in the proof of Proposition \ref{usce}.
As $v^k\in K(X\times [\al_k,1])$,
by Proposition \ref{ucen}, (b-\rmn4)
is equivalent to
(b-\rmn5).
As
$v^{k}\in K(X\times [\al_k, 1])\subset K(X\times [0, 1])$,
by Proposition \ref{bpe}(\rmn3), (b-\rmn6) is true.
Thus (b) is proved.

Define a subset $v$ of $X \times [0,1]$
given by
\begin{equation}\label{fve}
v = \bigcup_{k=1}^{+\infty} v^k \cup ( X\times \{0\}).
\end{equation}
Let $k\in \mathbb{N}$.
Then from affirmation (b),
$\langle v \rangle_\al=  \langle v^k \rangle_\al$ for each $\al\in (\al_{k}, 1]$. This means that
\begin{gather}\label{vceg}
 v \cap (X\times (\al_k, 1]) = v^{k} \cap (X\times (\al_k, 1]) \subseteq v^k.
   \end{gather}
So
\begin{equation}\label{vce}
   \langle v\rangle_\al =
\left\{
  \begin{array}{ll}
 \langle v^k \rangle_\al,  & \mbox{ if }  k\in \mathbb{N} \hbox{ and }  \al\in (\al_k,1],\\
 X, &  \mbox{ if } \al=0,
  \end{array}
\right.
\end{equation}

We show that
$v\in C(X\times [0,1])$. To this end,
let $\{(x_l, \gamma_l)\}$ be a sequence in $v$
which converges to an element $(x,\gamma)$
in $X\times [0,1]$. If $\gamma=0$, then clearly $(x,\gamma) \in v$. If $\gamma>0$,
then there is a $k_0\in \mathbb{N}$ such that
$\gamma > \al_{k_0}$. Hence there is an $L$ such that $\gamma_l > \al_{k_0}$ when $l\geq L$.
So by \eqref{vceg},
 $(x_l, \gamma_l) \in v^{k_0}$ when $l\geq L$.
Since
$v^{k_0} \in K(X\times [\al_{k_0}, 1])$, it follows that $(x,\gamma) \in v^{k_0}\subset v$.

We claim that
\begin{description}
  \item[(c)] $\lim_{n\to\infty} H({\rm end}\, u_n^{(n)},  v) =0$ and $v\in F_{USCG} (X)^e$.
\end{description}

Let $n,k\in \mathbb{N}$.
Obviously $S_{{\rm end}_0^{\al_k} u_n^{(n)}} \leq \al_k$.
As $\langle {\rm end}_0^{\al_k}\, u_n^{(n)}\rangle_{\al_k} = [u_n^{(n)}]_{\al_k}\supseteq [u_n^{(n)}]_\xi \not= \emptyset$, it follows that $S_{{\rm end}_0^{\al_k} u_n^{(n)}} = \al_k$.
Clearly
$S_{v \cap (X\times [0, \al_k])} \leq \al_k$. So
\begin{gather}
 H^* ({\rm end}_0^{\al_k} \, u_n^{(n)},  v) \leq H^* ({\rm end}_0^{\al_k} \, u_n^{(n)}, X\times\{0\})
=\al_k, \label{fem}
\\
 H^*( v \cap (X\times [0, \al_k]),\, {\rm end}\, u_n^{(n)})
\leq H^* (v \cap (X\times [0, \al_k]), X\times\{0\})
\leq \al_k.\label{pem}
\end{gather}
(Let $k\in \mathbb{N}$. As $v^k\not=\emptyset$, by (b-\rmn4), $\langle v^k\rangle_{\al_k}=\bigcup_{\gamma\in [\al_k,1]} \langle v^k\rangle_{\gamma} \not= \emptyset$.
Then by \eqref{fve}, $\langle v\rangle_{\al_k} \supseteq \langle v^k\rangle_{\al_k} \not= \emptyset$.
Hence
$\langle v \cap (X\times [0, \al_k]) \rangle_{\al_k}= \langle v\rangle_{\al_k} \not=\emptyset$.
Thus $S_{v \cap (X\times [0, \al_k])} = \al_k$.
So
 the last ``$\leq$'' in \eqref{pem} can be replaced
by ``=''. )

By \eqref{fve} and \eqref{fem}, for each $n,k\in \mathbb{N}$,
\begin{align}\label{hle}
  H^* ({\rm end}\, u_n^{(n)},  v) &= \max\{  H^* ({\rm end}_{\al_k} \, u_n^{(n)},  v) ,\
 H^* ({\rm end}_0^{\al_k} \, u_n^{(n)},  v)    \} \nonumber \\
&
\leq \max\{H^*({\rm end}_{\al_k}\, u_n^{(n)},\,  v^{k}),\ \al_k\}.
\end{align}
By \eqref{vceg} and \eqref{pem}, for each $n,k\in \mathbb{N}$,
\begin{align}\label{hre}
H^* ( v, {\rm end}\, u_n^{(n)}) &  = \max\{
\sup_{(x, \gamma) \in  v \cap (X\times (\al_k, 1]} \overline{d} ( (x, \gamma),  \, {\rm end}\, u_n^{(n)} ),
\ H^*( v \cap (X\times [0, \al_k]),\, {\rm end}\, u_n^{(n)})\} \nonumber
\\
& \leq \max\{H^* ( v^{k},\, {\rm end}_{\al_{k}}\, u_n^{(n)}),    \ \al_k\}.
\end{align}
Clearly,
\eqref{hle} and \eqref{hre} imply that for each $n,k\in \mathbb{N}$,
\begin{align}\label{hlre}
  H ({\rm end}\, u_n^{(n)}, \, v) & \leq     \max\{  H ({\rm end}_{\al_k} \, u_n^{(n)}, \,  v^{k}),\ \al_k\}.
\end{align}

Now we show that
\begin{equation}\label{lgce}
\lim_{n\to\infty} H ({\rm end}\, u_n^{(n)},  v) =0.
\end{equation}
To see this,
let $\varepsilon>0$.
Notice that $\al_k \to 0$ and for each $\al_k$, $k=1,2,\ldots$,
$\lim_{n\to\infty} H({\rm end}_{\al_k}\, u_n^{(n)},   v^{k}) = 0$.
Then there is an $\al_{k_0}$ and an $N\in \mathbb{N}$ such that
 $\al_{k_0} < \varepsilon$
and
$ H({\rm end}_{\al_{k_0}}\, u_n^{(n)},   v^{k_0}) <\varepsilon$ for all $n\geq N$.
Thus by \eqref{hlre},
$H ({\rm end}\, u_n^{(n)},  v)<\varepsilon$ for all $n\geq N$. So \eqref{lgce} is true.

Since the sequence $\{ {\rm end}\, u_n^{(n)} \}$ is in $F_{USC}(X)^e$
and
$\{ {\rm end}\, u_n^{(n)} \}$ converges to $v$ in $(C(X \times [0,1]), H)$,
by
Proposition \ref{usce}, it follows that $v\in F_{USC}(X)^e$.
To show $v\in F_{USCG}(X)^e$, we only need
to
show that for each $\al\in (0,1]$, $\langle v\rangle_\al \in K(X) \cup \{\emptyset\}$.

Let $\al\in (0,1]$. Then
there is a $k\in \mathbb{N}$ such that
 $\al\in (\al_k,1]$.
 Thus
by \eqref{vce} and (b-\rmn6), $\langle v\rangle_\al =   \langle v^k \rangle_\al \in K(X) \cup \{\emptyset\}$. Hence $v\in F_{USCG}(X)^e$.
So affirmation (c) is proved.

From
 affirmation (c), we have
that
 $\{u_n^{(n)}\}$ is a convergent sequence in $(F_{USCG} (X), H_{\rm end})$.
Note that $\{u_n^{(n)}\}$ is a subsequence of $\{u_n\}$.
Thus the proof is completed.

\end{proof}

\begin{re}
{\rm
We make some additions for Theorem \ref{rcfegnum} and its proof.
Below, we continue to use the notations used for the proof.

(\uppercase\expandafter{\romannumeral1})
We can also show that $v\in F_{USCG}(X)^e$ as follows.

By Proposition \ref{repu},
to show $v\in F_{USCG}(X)^e$, it suffices
to
show that (\rmn1) for each $\al\in(0,1]$, $\langle v\rangle_\al \in K(X) \cup \{\emptyset\}$; (\romannumeral2) \ $\langle v\rangle_0 = X $;
(\rmn3)
 for each $\al\in(0,1]$, $\langle v\rangle_\al = \bigcap_{\beta < \al} \langle v \rangle_\beta$.

(\rmn1) is shown (see the penultimate paragraph of the proof of Theorem
\ref{rcfegnum}). (\rmn2) is obvious.
Now we show (\rmn3). First we
verify the following property (f):
\\
(f) Let $\al,\beta\in [0,1]$ with $\beta<\al$. Then $\langle v \rangle_\al \subseteq \langle v \rangle_\beta$.

If $\beta=0$ then $\langle v \rangle_\beta= X $ and hence $\langle v \rangle_\al \subseteq \langle v \rangle_\beta$.
If $\beta>0$ then there is an $\al_k$ such that $\al, \beta \in (\al_k, 1]$.
From \eqref{vceg} and (b-\rmn4), it follows that $\langle v \rangle_\al
=\langle v^k \rangle_\al  \subseteq \langle v^k \rangle_\beta = \langle v \rangle_\beta$.
So (f) is true.

Let $\al\in(0,1]$. Then there is an $\al_k$ such that $\al\in (\al_k, 1]$.
So by \eqref{vceg}, (b-\rmn5) and (f),
$\langle v\rangle_\al = \langle v^k\rangle_\al
= \bigcap_{\beta\in [\al_k, \al)} \langle v^k\rangle_\beta
= \bigcap_{\beta < \al} \langle v \rangle_\beta$.
Thus (\rmn3) is proved.
So $v\in F_{USCG}(X)^e$.

(\uppercase\expandafter{\romannumeral2})
We can also show that for each $k\in \mathbb{N}$, $S_{v \cap (X\times [0, \al_k])} = \al_k$ by using affirmation (c).

Let $k\in \mathbb{N}$. As $\lim_{n\to\infty} H({\rm end}\, u_n^{(n)},  v) =0$ and for each $n\in \mathbb{N}$, $S_{{\rm end}\, u_n^{(n)}}=S_{u_n^{(n)}} \geq \xi$,
by Proposition \ref{semg}, $S_{v}\geq \xi>0$. Then $v\not=\emptyset$.
As $v\in F_{USCG} (X)^e \setminus \{\emptyset\}$ and $\al_k \in [0,\xi] \subseteq [0,S_v]$, by Proposition \ref{mar}(\rmn2-2),
$\langle v \cap (X\times [0, \al_k]\rangle_{\al_k} =  \langle v \rangle_{\al_k} \not=\emptyset$.
Thus $S_{v \cap (X\times [0, \al_k]}=\al_k$.

(\uppercase\expandafter{\romannumeral3})
We claim that
\begin{description}
  \item[(d)]
(d-\rmn1) For each $k\in \mathbb{N}$, $S_v= S_{v^k} \geq \xi$.
\\
(d-\rmn2)
For each $k\in \mathbb{N}$ with $k>1$,
$\langle v\rangle_{S_v} = \langle v^k\rangle_{S_{v^k}}$.
\\
(d-\rmn3) Let $k\in \mathbb{N}$. Then for each $\al\in [\al_k, S_{v^k}]$, $\langle v^k \rangle_{\al} \in K(X)$; for each $\al\in [0,1]\setminus [\al_k, S_{v^k}]$, $\langle v^k \rangle_{\al}=\emptyset$.
\end{description}

Observe that
for each $n,k\in \mathbb{N}$,
$S_{{\rm end}_{\al_k}\, u_{n}^{(k)}} = S_{u_{n}^{(k)}} \geq \xi$.
Let $k\in \mathbb{N}$.
As $H({\rm end}_{\al_k}\, u_{n}^{(k)},  v^k ) \to 0$, by Proposition \ref{semg},
$S_{v^k} = \lim_{n\to\infty} S_{{\rm end}_{\al_k}\, u_{n}^{(k)}}\geq \xi$.

Let $k\in \mathbb{N}$ with $k>1$. Since
$\{u_{n}^{(k)} \}$ is a subsequence of $\{u_{n}^{(1)} \}$,
we have that
$S_{v^k} = \lim_{n\to\infty} S_{u_{n}^{(k)}} =  \lim_{n\to\infty} S_{u_{n}^{(1)}} = S_{v^1}$.
Note that
$\al_k< \xi\leq 1$.
Thus
by \eqref{vceg}, $S_v= S_{v^k}$ and $\langle v\rangle_{S_v} = \langle v^k\rangle_{S_{v^k}}$.
So (d-\rmn1) and (d-\rmn2) are true.

Let $k\in \mathbb{N}$.
As $v^k\in K(X\times [\al_k,1])$, by (b-\rmn4) and Proposition \ref{bpue}(\rmn3), and by Remark \ref{bre}(\rmn1),
(d-\rmn3) is true.
So (d) is proved.
}
\end{re}

\begin{tm} \label{tbfegnu}
  Let $U$ be a subset of $F_{USCG} (X)$. Then $U$ is totally bounded in $(F_{USCG} (X), H_{\rm end})$
if and only if
$U(\al)$
is totally bounded in $(X,d)$ for each $\al \in (0,1]$.
\end{tm}

\begin{proof}
 \textbf{\emph{Necessity}}.
The necessity part is Lemma \ref{tbfegnum}.

\textbf{\emph{Sufficiency}}.
Suppose that
$U(\al)$ is totally bounded
in $(X, d)$ for each $\al \in (0,1]$.
Then
$U(\al)$ is relatively compact
in $(\widetilde{X}, \widetilde{d})$ for each $\al \in (0,1]$.
Thus by Theorem \ref{rcfegnum},
$U$ is
 relatively compact in
 $(F_{USCG} (\widetilde{X}), H_{\rm end})$.
Hence
$U$ is
 totally bounded in
 $(F_{USCG} (\widetilde{X}), H_{\rm end})$.
So clearly
$U$
is
 totally bounded in
 $(F_{USCG} (X), H_{\rm end})$.

\end{proof}

\begin{tm}\label{cfegum}
  Let $U$ be a subset of $F_{USCG} (X)$. Then the following are equivalent:
\\
(\rmn1)
 $U$ is compact in $(F_{USCG} (X), H_{\rm end})$;
\\
(\rmn2)  $U(\al)$
is relatively compact in $(X, d)$ for each $\al \in (0,1]$ and $U$ is closed in $(F_{USCG} (X), H_{\rm end})$;
\\
(\rmn3)
  $U(\al)$
is compact in $(X, d)$ for each $\al \in (0,1]$ and $U$ is closed in $(F_{USCG} (X), H_{\rm end})$.

\end{tm}

\begin{proof}
By Theorem \ref{rcfegnum},
  (\romannumeral1) $\Leftrightarrow$ (\romannumeral2).
  Obviously (\romannumeral3) $\Rightarrow$ (\romannumeral2).
We
shall complete the proof by showing that
 (\romannumeral2) $\Rightarrow$
 (\romannumeral3).
 Suppose that (\romannumeral2) is true.
To verify (\romannumeral3),
it suffices to
show that $U(\al)$ is closed in $(X,d)$ for each $\al\in (0,1]$.
To do this,
let $\al\in (0,1]$
and
let $\{x_n\}$ be a sequence in $U(\al)$ with $\{x_n\}$ converges to an element $x$ in $(X,d)$.
We only need to
show that $x\in U(\al)$.

Pick
a sequence $\{u_n\}$ in $U$
such that
$x_n \in [u_n]_\al$ for $n=1,2,\ldots$, which means
 that $(x_n,\al) \in {\rm end}\, u_n$
for $n=1,2,\ldots$.

From the equivalence of (\romannumeral1) and (\romannumeral2),
 $U$ is compact in $(F_{USCG} (X), H_{\rm end})$.
So
there exists a subsequence $\{u_{n_k}\}$ of $\{u_n\}$ and $u\in U$
such that $H_{\rm end}(u_{n_k}, u) \to 0$.
Hence
by Remark \ref{hmr}, $\lim_{n\to \infty}^{(\Gamma)}  u_{n_k} = u $.
Note that $(x,\al) = \lim_{k\to\infty}(x_{n_k}, \al)$.
Thus
$$(x,\al) \in \liminf_{n\to \infty}{\rm end}\, u_{n_k}
=
{\rm end}\, u.$$
So
$x\in [u]_\al$, and therefore $x\in U(\al)$.

It can be seen that
Theorem 7.10 in \cite{huang719} can be verified in a similar fashion to this theorem.

\end{proof}

We \cite{huang} gave the following characterizations of compactness in $(F_{USCG}(\mathbb{R}^m), H_{\rm end})$.

\begin{tm}\label{fuscbrcn}
(Theorem 7.1 in \cite{huang})
Let $U$ be a subset of $F_{USCG}(\mathbb{R}^m)$.
Then
  $U$ is a relatively compact set in $(F_{USCG}(\mathbb{R}^m), H_{\rm end})$
  if and only if $U(\al)$ is a bounded set in $\mathbb{R}^m$ when $\al \in (0,1]$.
\end{tm}

\begin{tm} \label{fusbtoundcn} (Theorem 7.3 in \cite{huang})
Let $U$ be a subset of $ F_{USCG}(\mathbb{R}^m) $.
Then
  $U$ is a totally bounded set in $(F_{USCG}(\mathbb{R}^m), H_{\rm end})$
  if and only if, for each $\al \in (0,1]$,
 $ U(\al) $ is a bounded set in $\mathbb{R}^m $.

\end{tm}

\begin{tm} \label{fuscbcn}(Theorem 7.2 in \cite{huang})
$U$ is a compact set in $(F_{USCG}(\mathbb{R}^m), H_{\rm end})$
  if and only if $U$ is a closed set in $(F_{USCG}(\mathbb{R}^m), H_{\rm end})$
  and
   $ U(\al) $ is a bounded set in $\mathbb{R}^m $
 when $\al \in (0,1]$.
\end{tm}

 Let $S$ be a set in $\mathbb{R}^m$.
Then the following properties are equivalent.
\\
(\romannumeral1) \ $S$ is a bounded set in $\mathbb{R}^m $.
\\
(\romannumeral2) \ $S$ is a totally bounded set in $\mathbb{R}^m $.
\\
(\romannumeral3) \ $S$ is a relatively compact set in $\mathbb{R}^m $.

Using the above well-known fact, we can see
that Theorem \ref{rcfegnum} implies Theorem \ref{fuscbrcn};
Theorem \ref{tbfegnu} implies Theorem \ref{fusbtoundcn};
Theorem \ref{cfegum} implies Theorem \ref{fuscbcn}.

So
the
characterizations of
relative compactness,
total boundedness, and compactness in
$(F_{USCG} (\mathbb{R}^m), H_{\rm end})$ given in our previous work \cite{huang}
are corollaries of
the characterizations of relative compactness,
total boundedness, and compactness of $(F_{USCG} (X), H_{\rm end})$
given in this section, respectively.

Furthermore,
the characterizations of relative compactness,
total boundedness, and compactness of $(F_{USCG} (X), H_{\rm end})$
given in this section illustrate
the relationship between relative compactness,
total boundedness, and compactness of a set in $F_{USCG} (X)$
and that of the union of its elements' $\al$-cuts.

From above discussions, we can see
that
the characterizations of relative compactness,
total boundedness, and compactness of $(F_{USCG} (X), H_{\rm end})$
given in this section significantly
improve
the
characterizations of
relative compactness,
total boundedness, and compactness in
$(F_{USCG} (\mathbb{R}^m), H_{\rm end})$ given in our previous work \cite{huang}.

\begin{re}
  {\rm
The following clauses (\romannumeral1) and (\romannumeral2) are
pointed out
in Remark 5.1 of chinaXiv:202107.00011v2 (we submitted it on 2021-07-22).
\\
(\romannumeral1) $(F^1_{USCG}(X), H_{\rm end})$ can be treated as a subspace of $(C(X \times [0, 1]), H)$ by seeing
each $u \in F^1_{USCG}(X)$ as its endograph.
 \\
(\romannumeral2)  We can discuss the properties of $(F^1_{USCG}(X), H_{\rm end})$ by treating $(F^1_{USCG}(X), H_{\rm end})$ as a subspace of $(C(X \times [0, 1]), H)$.
These properties include characterizations of
total boundedness, relative compactness and compactness of $(F^1_{USCG}(X), H_{\rm end})$.

In this paper, we treat $(F_{USCG}(X), H_{\rm end})$ as a subspace of $(C(X \times [0, 1]), H)$ to discuss the properties of $(F_{USCG}(X), H_{\rm end})$.

  }
\end{re}

At the end of this section, we illustrate that
Theorems
\ref{rce}, \ref{tbe} and \ref{come} can be seen as special cases of
Theorems \ref{rcfegnum},
\ref{tbfegnu},
and \ref{cfegum}, respectively. We begin with some propositions.

The following proposition follows immediately from the basic definitions.

\begin{pp}\label{brn}
Let $A$ be a subset of $X$.
\\
(\romannumeral1)
 The conditions
(\romannumeral1-1) $A$ is a set in $C(X)$,
and
(\romannumeral1-2) $\chi_{A}$ is a fuzzy set in $F_{USC}(X)$,
are equivalent.
\\
(\romannumeral2)
The conditions
(\romannumeral2-1)
$A$ is a set in $K(X)$, (\romannumeral2-2) $\chi_{A}$ is a fuzzy set in $F_{USCB}(X)$, and (\romannumeral2-3) $\chi_{A}$ is a fuzzy set in $F_{USCG}(X)$,
are equivalent.

\end{pp}

Let $\mathcal{D}\subseteq P(X)$. We use the symbol $\mathcal{D}_{F(X)}$ to denote the set $\{C_{F(X)}: C\in \mathcal{D} \}$.

 Let $A, B\in C(X)$. Then ${\rm end}\, \chi_A = (A\times[0,1]) \cup (X\times\{0\})$ and
  \begin{gather}
H^*({\rm end}\, \chi_A,  {\rm end}\, \chi_B)= \min\{H^*(A,B), \ 1 \},\label{efm}
\\
   H_{\rm end} (\chi_A, \chi_B) = \min\{H(A,B), \ 1 \}.\label{efn}
  \end{gather}
Let $(x,\al)\in X\times[0,1]$. Then
$\overline{d}((x,\al),  {\rm end}\, \chi_B) =  \overline{d}((x,\al),  B\times[0,1]) \wedge\overline{d}((x,\al),  X\times\{0\})=  \overline{d}((x,\al),  B\times\{\al\}) \wedge\overline{d}((x,\al),  X\times\{0\})
= d(x, B) \wedge \al$.
Clearly if $(x,\al)\in X\times\{0\}$ then
$\overline{d}((x,\al),  {\rm end}\, \chi_B)=0$.
Thus
$H^*({\rm end}\, \chi_A,  {\rm end}\, \chi_B)= \sup\{\overline{d}((x,\al),  {\rm end}\, \chi_B): (x,\al) \in {\rm end}\, \chi_A  \}
=\sup\{d(x, B) \wedge \al: (x,\al) \in A\times[0,1] \}
=\sup\{d(x, B) \wedge 1 : x  \in A  \} =  \min\{H^*(A,B), \ 1 \}.$
So \eqref{efm} is true.

By \eqref{efm},
$H_{\rm end} (\chi_A, \chi_B) = H^*({\rm end}\, \chi_A,  {\rm end}\, \chi_B) \vee H^*({\rm end}\, \chi_B,  {\rm end}\, \chi_A) =  \min\{H^*(A,B), \ 1 \} \vee \min\{H^*(B,A), \ 1 \} = \min\{H^*(A,B)\vee H^*(B,A), \ 1 \} =   \min\{H(A,B), \ 1 \}$. So \eqref{efn} is true.

Here we mention that
 the above two paragraphs including \eqref{efm} and \eqref{efn} remain true if $\overline{d}$ is replaced by $d'$,
$H^*$ on $C(X\times [0,1])\cup\{\emptyset\}$ is replaced
 by
$H'^*$ on $C(X\times [0,1])\cup\{\emptyset\}$, and
$H_{\rm end}$ is replaced
 by $H'_{\rm end}$.

\begin{pp}\label{mrn}
 Let $\mathcal{D}$ be a subset of $K(X)$.
 \\
 (\romannumeral1) $\mathcal{D}$ is totally bounded in $(K(X), H)$
if and only if
$\mathcal{D}_{F(X)}$ is totally bounded in $(F_{USCG}(X), H_{\rm end})$;
\\
(\romannumeral2)
$\mathcal{D}$ is compact in $(K(X), H)$
if and only if
$\mathcal{D}_{F(X)}$ is compact in $(F_{USCG}(X), H_{\rm end}).$
\end{pp}

\begin{proof}

   From \eqref{efn}, it follows immediately that (\romannumeral1) is true.

    By \eqref{efn}, we have
     that
     $\mathcal{D}$ is compact in $(K(X), H)$
     if and only if $\mathcal{D}_{F(X)}$ is compact in $(K(X)_{F(X)}, H_{\rm end}).$
     Clearly $\mathcal{D}_{F(X)}$ is compact in $(K(X)_{F(X)}, H_{\rm end})$
     if and only if
     $\mathcal{D}_{F(X)}$ is compact in $(F_{USCG}(X), H_{\rm end}).$
   So (\romannumeral2)
is true.

\end{proof}

\begin{pp}\label{spcmgn}

Let $\{A_n\}$ be a sequence of sets in $C(X)$. If $\{\chi_{A_n}\}$ converges to a fuzzy set $u$ in $F_{USC} (X)$
according to the $H_{\rm end}$ metric,
then
there is an $A\in C(X)$ such that $u=\chi_A$ and $H(A_n, A)\to 0$ as $n\to\infty$.

\end{pp}

\begin{proof}

We will show in turn, the following properties (\romannumeral1), (\romannumeral2)  and (\romannumeral3).
\\
(\romannumeral1) Let $x\in X$ and $\al, \beta \in (0,1]$. Then $(x,\al) \in {\rm end}\, u$ if and only if $(x,\beta) \in {\rm end}\, u$.
\\
(\romannumeral2) $[u]_\al = [u]_\beta$
for all $\al,\beta\in [0,1]$.
\\
(\romannumeral3)
There is an $A$ in $C(X)$ such that $u=\chi_A$ and $H(A_n, A)\to 0$ as $n\to\infty$.

To show (\romannumeral1), we only need to show
that if
$(x, \alpha)\in {\rm end}\, u$ then
$(x, \beta)\in {\rm end}\, u$ since $\al$ and $\beta$ can be interchanged.

Assume that $(x,\al) \in {\rm end}\, u$.
Since $H_{\rm end} (\chi_{A_n}, u) \to 0$,
by
Theorem \ref{hkg} and Remark \ref{hmr},
 $\lim_{n\to\infty}^{(\Gamma)} \chi_{A_n} = u$. Then there is a
sequence
$\{(x_n, \al_n)\}$ such that $(x_n, \al_n) \in {\rm end}\, \chi_{A_n}$
for $n=1,2,\ldots$, and
$\lim_{n\to\infty} \overline{d} ( (x_n,\al_n), (x,\al)   )  = 0$.
As $\al>0$, it follows that there exists an
$N$ such that
$\al_n > 0$ for all $n\geq N$.
This yields
that $(x_n, \al_n) \in {\rm send}\, \chi_{A_n} = A_n\times [0,1]$ for all $n\geq N$.
Hence
$(x_n, \beta) \in {\rm send}\, \chi_{A_n}$ for all $n\geq N$.
Observe that $\lim_{n\to\infty} \overline{d} ((x_n,\beta), (x, \beta)   )  = 0$, i.e.
$\{(x_n,\beta): n\geq N\}$ converges to $(x,\beta)$ in $(X\times[0,1], \overline{d})$.
Thus
we have
$(x, \beta)\in \liminf_{n\to\infty} {\rm end}\, \chi_{A_n} = {\rm end}\, u$.
So (\romannumeral1) is true.

From (\romannumeral1), we have that
$[u]_\al = [u]_\beta$
for all $\al,\beta\in (0,1]$.
Then
$[u]_0 = \overline{\cup_{\al>0}[u]_{\al}} = [u]_1$.
So (\romannumeral2) is true.

Set $A=[u]_1$. Then $A\in C(X)\cup\{\emptyset\}$.
By Proposition \ref{sem}, $u\in F^{'1}_{USC}(X)$.
From this and (\romannumeral2), $A=[u]_{0.9} \not= \emptyset$, and hence
$A\in C(X)$. By (\rmn2), $u=\chi_A$.

Since by \eqref{efn},
$$H_{\rm end} (\chi_{A_n}, u) = H_{\rm end} (\chi_{A_n}, \chi_A) = \min\{   H({A_n}, A),\ 1 \}\to 0    \mbox{ as }    n\to \infty,$$
we obtain that
$H({A_n}, A) \to 0$ as $n\to \infty$.
So (\romannumeral3) is true.
This completes the proof.

\end{proof}

\begin{pp}\label{spcm}

Let $\{x_n\}$ be a sequence in $X$. If $\{\widehat{x_n}\}$ converges to a fuzzy set $u$ in $F_{USC} (X)$
according to the $H_{\rm end}$ metric,
then
there is an $x\in X$ such that $u=\widehat{x}$ and $d(x_n, x)\to 0$ as $n\to\infty$.

\end{pp}

\begin{proof} Note that $\widehat{z} = \chi_{\{z\}}$ for each $z\in X$.
 Thus by Proposition \ref{spcmgn}, there is an
$A\in C(X)$ such that $u=\chi_A$ and $H(\{x_n\}, A)\to 0$ as $n\to\infty$.
  Since $\lim^{(K)}_{n\to\infty} \{x_n\} = A$,
  it follows that $A$ is a singleton. Set $A=\{x\}$.
  Then $u=\widehat{x}$ and $d(x_n, x) = H(\{x_n\}, \{x\}) \to 0$ as $n\to\infty$. This completes the proof.

 \end{proof}

   Proposition \ref{spcm} is Proposition 8.15 in our paper arXiv:submit/4644498.
  It can be seen that we can also use the idea in the proof of Proposition \ref{spcmgn}
   to show Proposition \ref{spcm} directly.

 It can be seen that using the idea in the proof of Proposition 8.15 in  arXiv:submit/4644498, we can show that $A$ in the proof of Proposition \ref{spcm} is a singleton as follows.

    Assume that $A$ has at least two distinct elements.
    Pick $p,q$ in $A$ with $p \not= q$.
Let $z\in X$. Since $d(p,z) + d(q,z) \geq d(p,q)$, it follows that
$\max\{d(p,z), d(q,z)\} \geq \frac{1}{2} d(p,q)$.
Thus $H(A, \{x_n\}) = H^*(A, \{x_n\}) \geq \frac{1}{2}d(p,q)$,
which contradicts
 $H(A, \{x_n\})\to 0$ as $n\to \infty$.

\begin{pp}\label{crpn}
Let $\mathcal{D}$ be a subset of $K(X)$ and $\mathcal{B}$ a subset of $C(X)$.
\\
 (\romannumeral1) $C(X)_{F(X)}$ is closed in $(F_{USC}(X), H_{\rm end}).$
\\
 (\romannumeral2) $K(X)_{F(X)}$ is closed in $(F_{USCG}(X), H_{\rm end}).$
  \\
 (\romannumeral3) $\mathcal{D}$ is closed in $(K(X), H)$
 if and only if $\mathcal{D}_{F(X)}$ is closed in $(F_{USCG}(X), H_{\rm end}).$
 \\
 (\romannumeral4)
$\mathcal{D}$ is relatively compact in $(K(X), H)$
if and only if
$\mathcal{D}_{F(X)}$ is relatively compact in $(F_{USCG}(X), H_{\rm end}).$
\\
 (\romannumeral5) $\mathcal{B}$ is closed in $(C(X), H)$
 if and only if $\mathcal{B}_{F(X)}$ is closed in $(F_{USC}(X), H_{\rm end}).$
 \\
 (\romannumeral6)
$\mathcal{B}$ is relatively compact in $(C(X), H)$
if and only if
$\mathcal{B}_{F(X)}$ is relatively compact in $(F_{USC}(X), H_{\rm end}).$
\end{pp}

\begin{proof}

From Proposition \ref{spcmgn} we have that (\romannumeral1) is true.

By (\romannumeral1),
the closure of $K(X)_{F(X)}$ in $(F_{USCG}(X), H_{\rm end})$ is contained in $F_{USCG}(X)\cap C(X)_{F(X)}$. From Proposition \ref{brn} (\romannumeral2),
 $F_{USCG}(X)\cap C(X)_{F(X)}=K(X)_{F(X)}$.
Thus the closure of $K(X)_{F(X)}$ in $(F_{USCG}(X), H_{\rm end})$
is $K(X)_{F(X)}$.
So (\romannumeral2) is true.

Suppose the following conditions:
(a-1) $\mathcal{D}$ is closed in $(K(X), H)$,
(a-2) $\mathcal{D}_{F(X)}$ is closed in $(K(X)_{F(X)}, H_{\rm end})$, and
(a-3) $\mathcal{D}_{F(X)}$ is closed in $(F_{USCG}(X), H_{\rm end}).$

  By \eqref{efn}, (a-1)$\Leftrightarrow$(a-2).
 From (\romannumeral2), (a-2)$\Leftrightarrow$(a-3).
Thus (a-1)$\Leftrightarrow$(a-3). So (\romannumeral3) is true.

Suppose the following conditions:
(b-1) $\mathcal{D}$ is relatively compact in $(K(X), H)$,
(b-2) $\mathcal{D}_{F(X)}$ is relatively compact in $(K(X)_{F(X)}, H_{\rm end})$, and
(b-3) $\mathcal{D}_{F(X)}$ is relatively compact in $(F_{USCG}(X), H_{\rm end}).$

  By \eqref{efn}, (b-1)$\Leftrightarrow$(b-2).
 From (\romannumeral2), (b-2)$\Leftrightarrow$(b-3).
Thus (b-1)$\Leftrightarrow$(b-3). So (\romannumeral4) is true.

Using \eqref{efn} and (\romannumeral1),
 (\romannumeral5) and (\romannumeral6) can be proved
in a similar manner to
 (\romannumeral3) and (\romannumeral4), respectively.

\end{proof}

Each subset $\mathcal{D}$ of $(K(X), H)$ corresponds a subset
$\mathcal{D}_{F(X)}$ of $(F_{USCG}(X), H_{\rm end}).$
Using
 Theorems \ref{rcfegnum},
\ref{tbfegnu},
and \ref{cfegum},
we obtain
the
characterizations of
relative compactness,
total boundedness, and compactness for $\mathcal{D}_{F(X)}$ in
$(F_{USCG} (X), H_{\rm end})$ as follows.

\begin{tl}
 \label{trcfegnum}
   Let $\mathcal{D}$ be a subset of $K(X)$. Then $\mathcal{D}_{F(X)}$ is relatively compact in $(F_{USCG} (X), H_{\rm end})$
if and only if
$\mathbf{ D} =   \bigcup \{C:  C \in  \mathcal{D} \}$
is relatively compact in $(X, d)$.

\end{tl}

\begin{tl} \label{ctybfegnu}
  Let $\mathcal{D}$ be a subset of $K(X)$. Then $\mathcal{D}_{F(X)}$ is totally bounded in $(F_{USCG} (X), H_{\rm end})$
if and only if
$\mathbf{ D} =   \bigcup \{C:  C \in  \mathcal{D} \}$
is totally bounded in $(X,d)$.
\end{tl}

\begin{tl}\label{crfegum}
  Let $\mathcal{D}$ be a subset of $K(X)$. Then the following are equivalent:
\\
 (\romannumeral1)
 $\mathcal{D}_{F(X)}$ is compact in $(F_{USCG} (X), H_{\rm end})$;
\\
 (\romannumeral2) $\mathbf{ D} =   \bigcup \{C:  C \in  \mathcal{D} \} $
is relatively compact in $(X, d)$ and $\mathcal{D}_{F(X)}$ is closed in $(F_{USCG} (X), H_{\rm end})$;
\\
 (\romannumeral3) $\mathbf{ D} =   \bigcup \{C:  C \in  \mathcal{D} \} $
is compact in $(X, d)$ and $\mathcal{D}_{F(X)}$ is closed in $(F_{USCG} (X), H_{\rm end})$.

\end{tl}

From Proposition \ref{mrn} and clauses (\romannumeral3) and (\romannumeral4) of Proposition \ref{crpn},
we obtain that Corollaries \ref{trcfegnum}, \ref{ctybfegnu} and \ref{crfegum}
are equivalent forms of
Theorems
 \ref{rce}, \ref{tbe} and \ref{come}, respectively.
So we can see
Theorems
\ref{rce}, \ref{tbe} and \ref{come} as special cases of
Theorems \ref{rcfegnum},
\ref{tbfegnu},
and \ref{cfegum}, respectively.

\section{Characterizations of compactness in $(F^r_{USCG} (X), H_{\rm end})$}

In this section, we first investigate
 the properties of the $H_{\rm end}$ metric.
Then
based on the characterizations of relative compactness,
total boundedness and compactness in $(F_{USCG} (X), H_{\rm end})$
given in Section \ref{cmg},
we give
characterizations of
relatively compact sets,
totally bounded sets, and compact sets in
$(F^r_{USCG} (X), H_{\rm end})$, $r\in [0,1]$.
$(F^r_{USCG} (X), H_{\rm end})$, $r\in [0,1]$ are a kind of subspaces of $(F_{USCG} (X), H_{\rm end})$.
Each element
in $F^r_{USCG} (X)$ takes $r$ as its maximum value.
$(F^1_{USCG} (X), H_{\rm end})$
is one of these subspaces.

We claim that
for $D, E \in C(X \times [0,1])$,
\begin{equation}\label{srhenu}
  H (D, E) \geq |S_D - S_E|.
\end{equation}
To see this,
let $D,E \in C(X \times    [0,1])$.
If $|S_D - S_E|=0$, then  \eqref{srhenu} is true. If $|S_D - S_E|>0$. Assume that $S_D > S_E$.
Note that for each $(x, t) \in D$ with $t> S_E$,
$\overline{d} ((x, t) , E) \geq t - S_E$. Thus
$H (D, E) \geq \sup   \{     t - S_E : (x, t) \in D \mbox{ with } t> S_E \}= S_D - S_E$.
So \eqref{srhenu} is true.

\begin{pp}\label{semg}
   Let $C$ and $C_n$, $n=1,2,\ldots$, be sets in $C(X \times [0,1])$.
If $H (C_n, C) \to 0$ as $n\to \infty$,
 then $S_{C_n} \to S_C$ as $n\to \infty$.
\end{pp}

\begin{proof}
  The desired result follows immediately from \eqref{srhenu}.
\end{proof}

Let $u\in F(X)$.
Define $\bm{S_u} :=  \sup\{u(x): x\in X\}$. Then $S_u\in [0,1]$.
We can see that $S_u = S_{{\rm end}\, u}$.
Clearly for each $\al\in [0,1]\setminus [0,S_u]$, $[u]_\al=\emptyset$.
If $S_u>0$, then for each $\al\in [0,S_u)$, $[u]_\al\not=\emptyset$.
$[u]_{S_u} = \emptyset$ is possible. See \eqref{bsr} and Examples \ref{empu} and \ref{rnce}.

Let $u\in F(X)$.
It is easy to see that
\begin{equation}\label{bsr}
\mbox{$u=\emptyset_{F(X)} \Leftrightarrow [u]_0=\emptyset \Leftrightarrow$ for each $x\in X$, $u(x)=0  \Leftrightarrow$ $S_u=0$.}
  \end{equation}

From \eqref{srhenu}, we have that
for each $u,v \in F_{USC} (X)$,
\begin{equation}\label{srhe}
  H_{\rm end} (u, v) \geq |S_u - S_v|.
\end{equation}

\begin{pp}\label{sem}
  Let $u$ and $u_n$, $n=1,2,\ldots$, be fuzzy sets in $F_{USC}(X)$.
If $H_{\rm end} (u_n, u) \to 0$ as $n\to \infty$,
 then $S_{u_n} \to S_u$ as $n\to \infty$.
\end{pp}

\begin{proof}
The desired result follows immediately from \eqref{srhe}.
Clearly, the desired conclusion is a corollary of
Proposition \ref{semg}.
\end{proof}

Let $u,v \in F_{USC} (X)$.
Clearly
$
  H^*({\rm end}\, u, {\rm end}\,v) \leq H^*({\rm end}\, u, X\times\{0\})=S_u
$, and so
$
  H_{\rm end} (u, v) \leq \max\{S_u, S_v\}
$.

Here we mention that the contents
from ``we claim that" in this section to the end of the above paragraph
 including \eqref{srhenu} and \eqref{srhe} remain true if $\overline{d}$ is replaced by $d'$,
$H_{\rm end}$ is replaced
 by $H'_{\rm end}$, and $H^*$ and $H$ on $C(X\times [0,1])\cup\{\emptyset\}$ are replaced
 by
$H'^*$ and $H'$ on $C(X\times [0,1])\cup\{\emptyset\}$, respectively.

\begin{re}\label{bmer}
  {\rm
Let $u\in F(X)$. $\max\{u(x): x\in X\}$ may not exist.
Clearly the conditions (a) $\max\{u(x): x\in X\}$ exists, (b)
there is an $x$ with $u(x)=S_u$, and (c)
 $S_u=\max\{u(x): x\in X\}$, are equivalent. We claim the following:
\\
(\rmn1) If $[u]_{S_u} \not= \emptyset$, then
$S_u =  \max\{u(x): x\in X\}$ and $S_u>0$.
\\
(\rmn2)
If $S_u=0$ (i.e. $[u]_0=\emptyset$), then $u=\emptyset_{F(X)}$ and so
$\max\{u(x): x\in X\} = 0 = S_u$.
\\
(\rmn3)
Let $S_u>0$. Then $S_u =  \max\{u(x): x\in X\}$ if and only if $[u]_{S_u} \not= \emptyset$.

Assume that $[u]_{S_u} \not= \emptyset$. Then (b) holds, and hence (c) is true.
Furthermore $S_u>0$ since otherwise $S_u=0$ and then, by \eqref{bsr}, $[u]_{S_u} = \emptyset$, which is a contradiction.
 So (\rmn1) is true.   (\rmn2) is obvious.
We can see that
if $S_u>0$ and $S_u =  \max\{u(x): x\in X\}$ then $[u]_{S_u}=\{x: u(x)\geq S_u\} \not= \emptyset$.
From this and (\rmn1), we obtain (\rmn3).

}
\end{re}

\begin{pp} \label{mar}
(\romannumeral1)
Let $u\in F_{USC}(X)$.
Let $\al\in [0, S_u]$ with $[u]_\al \in K(X)$. Then
for each $\beta\in [\al, S_u]$,
$[u]_{\beta}  \in K(X)$. And
$S_u =  \max\{u(x): x\in X\}$.

(\romannumeral2)
Let $u \in F_{USCG} (X)\setminus \{\emptyset_{F(X)}\}$. Then
(\rmn2-1)
$S_u>0$, (\rmn2-2) for each
$\al\in [0, S_u]$, $[u]_\al \not=\emptyset$, and (\rmn2-3)
for each
$\al\in (0, S_u]$, $ [u]_\al \in K(X)$.

(\romannumeral3)
Let $u \in F_{USCG} (X)$. Then
$S_u =  \max\{u(x): x\in X\}$.
\end{pp}

\begin{proof}
First, we show (\romannumeral1).
As ${\rm end}\, u \in C(X\times [0,1])$,
by  Proposition \ref{bpue} (\rmn2) (put $c_0=0$, $c_1=1$ and $c_2=\al$ in Proposition \ref{bpue} (\rmn2)),
for each $\beta\in [\al, S_u]$,
$[u]_{\beta}  \in K(X)$.
Then $[u]_{S_u}  \not= \emptyset$
and so $S_u =  \max\{u(x): x\in X\}$. Thus (\romannumeral1) is true.

Another proof of (\rmn1) is as follows. Firstly we show that $[u]_{S_u} \not= \emptyset$.
If $\al = S_u$, then $[u]_{S_u} \not= \emptyset$.
If $\al< S_u$, then pick a sequence $\{x_n\}$ in $[u]_\al$ with $u(x_n) \to S_u$.
From the compactness of $[u]_\al$, there is a subsequence
$\{x_{n_k}\}$ of $\{x_n\}$ such that
$\{x_{n_k}\}$ converges to a point $x$ in $[u]_\al$.
Thus
$u(x) \geq \lim_{k\to\infty} u(x_{n_k}) = S_u$.
Hence
$u(x) = S_u$, and therefore $[u]_{S_u}  \not= \emptyset$.
So $S_u =  \max\{u(x): x\in X\}$.
Let $\beta\in [\al, S_u]$.
Then $[u]_\beta \supseteq [u]_{S_u}  \not= \emptyset$ and hence $[u]_\beta \in C(X)$.
Thus $[u]_{\beta}\in K(X)$ because $[u]_{\beta}$ is a nonempty closed subset of $[u]_\al$ and $[u]_\al\in K(X)$.
So (\romannumeral1) is true.

Now we show (\romannumeral2).
By \eqref{bsr}, (\rmn2-1) is true.
Note that for each $\al\in (0,1]$, $[u]_{\alpha}\in K(X)\cup\{\emptyset\}$,
and that
for each $\al\in [0, S_u)$, $[u]_{\alpha}\not=\emptyset$.
So $[u]_{S_u/2}\in K(X)$.
Then by (\rmn1),
$[u]_{S_u}\not=\emptyset$.
Hence (\rmn2-2) is true, and
thus (\rmn2-3) is true.
So (\romannumeral2) is proved.

Finally we show (\romannumeral3).
If $u \in F_{USCG} (X) \setminus \{\emptyset_{F(X)}\}$, then from (\romannumeral2-2), $[u]_{S_u}\not=\emptyset$ and hence $S_u =  \max\{u(x): x\in X\}$.
If $u = \emptyset_{F(X)}$, then
$S_u = 0= \max\{u(x): x\in X\}$.
So (\rmn3) is true.

\end{proof}

Let $r\in [0,1]$. Define
\begin{gather*}
 F^{'r}_{USC} (X) = \{ u\in F_{USC} (X):  r= S_u \},
\\
 F^r_{USC} (X) = \{ u\in F_{USC} (X):  r=\max\{u(x): x\in X\}\},
 \\
  F^{'r}_{USCG} (X) = \{ u\in F_{USCG}(X): r = S_u \},
\\
  F^r_{USCG} (X) = \{ u\in F_{USCG}(X):   r=\max\{u(x): x\in X\}\},
  \\
  F^{'r}_{USCB} (X) = \{ u\in F_{USCB}(X):  r= S_u\},
\\
  F^r_{USCB} (X) = \{ u\in F_{USCB}(X):   r=\max\{u(x): x\in X\}\}.
\end{gather*}
Let $r\in [0,1]$. We can see that
 $F^r_{USC} (X)\subseteq  F^{'r}_{USC} (X)$.
Clearly,
 $F^{'0}_{USC} (X)= F^0_{USC} (X)=F^0_{USCG} (X) = F^0_{USCB} (X) = \{\emptyset_{F(X)}\}$.

Let $u\in F_{USC}(X)$.
Assume $r\in [0,1]$. Then $u\in   F^r_{USC} (X) $ if and only if $r=\max\{u(x): x\in X\}=S_u$.
Assume $r\in (0,1]$. Then $u\in F^r_{USC} (X) $ if and only if
 $S_u=r$ and $[u]_{S_u} \not= \emptyset$;
$u\in F^{'r}_{USC} (X) \setminus F^r_{USC} (X) $ if and only if
 $S_u=r$ and $[u]_{S_u} = \emptyset$ (see also Remark \ref{bmer}(\rmn3)).

For $r\in (0,1]$,
$F^r_{USC} (X)\subsetneqq  F^{'r}_{USC} (X)$
is possible. See Examples \ref{empu} and \ref{rnce}.

\begin{eap}\label{empu} {\rm
Let $r\in (0,1]$. Define $u\in F_{USC}(\mathbb{R})$ by putting
\[
[u]_\al=\left\{
  \begin{array}{ll}
   (-\infty, -\frac{1}{r-\al}], & \al\in [0, r), \\
\emptyset, & \al\in [r,1].
  \end{array}
\right.
\]
Then $S_u=r$ and $[u]_{S_u} = \emptyset$. This means that
$u\in F^{'r}_{USC} (\mathbb{R})$ but $u\notin F^r_{USC}(\mathbb{R})$.
}
\end{eap}

\begin{pp}\label{ace}
  Let $r\in [0,1]$. Then
\\
(\romannumeral1) \ $F^{r}_{USCG}(X)=F^{'r}_{USCG}(X)$, $F^{r}_{USCB}(X)=F^{'r}_{USCB}(X)$,
\\
(\romannumeral2) \ $F^{'r}_{USC}(X)$ is a closed subset of $(F_{USC}(X), H_{\rm end})$,
\\
(\romannumeral3) \
$F^{r}_{USCG}(X)$ is a closed subset of $(F_{USCG}(X), H_{\rm end})$, and
\\
(\romannumeral4) \
$F^{r}_{USCB}(X)$ is a closed subset of $(F_{USCB}(X), H_{\rm end})$.
\end{pp}

\begin{proof}

We can see that $F^{r}_{USCG}(X)\subseteq F^{'r}_{USCG}(X)$.
Let $u\in F^{'r}_{USCG}(X)$. Then by Proposition \ref{mar} (\romannumeral3), $r=S_u=\max\{u(x): x\in X\}$. So
$u\in F^{r}_{USCG}(X)$. Since $u\in F^{'r}_{USCG}(X)$ is arbitrary,
we have that $F^{'r}_{USCG}(X)\subseteq F^{r}_{USCG}(X)$.
Thus $F^{'r}_{USCG}(X)= F^{r}_{USCG}(X)$.
 Hence
$F^{'r}_{USCB}(X)=F^{'r}_{USCG}(X)\cap F_{USCB}(X) =F^{r}_{USCG}(X)\cap F_{USCB}(X) = F^{r}_{USCB}(X)$. So (\rmn1) is true.

The fact that
$F^{'r}_{USCB}(X)= F^{r}_{USCB}(X)$ can also be shown in a similar fashion to the fact that $F^{'r}_{USCG}(X)= F^{r}_{USCG}(X)$

By Proposition \ref{sem}, (\romannumeral2) is true.

From Proposition \ref{sem}, $F^{'r}_{USCG}(X)$ is a closed subset of $(F_{USCG}(X), H_{\rm end})$, and
$F^{'r}_{USCB}(X)$ is a closed subset of $(F_{USCB}(X), H_{\rm end})$. It then
follows from (\romannumeral1) that (\romannumeral3) and (\romannumeral4) are true.

\end{proof}

 $F^{r}_{USC}(X)$ may not be a closed subset of $(F_{USC}(X), H_{\rm end})$.
The following is a such example.

\begin{eap}\label{rnce}
  {\rm
Let $Y=\mathbb{R}\setminus \{1\}$. We also use $Y$ to denote
the metric space $(Y,\rho)$, where the metric $\rho$ on $Y$ is defined as $\rho(x,y) = |x-y|$ for each $x,y\in Y$.

Let $r\in (0,1]$.
Let $n\in \mathbb{N}$. Define $u_n\in F_{USC}(Y)$ by putting
\[
[u_n]_\al=\left\{
  \begin{array}{ll}
    [1-\frac{1}{n}, 2-\frac{\al}{r}] \setminus \{1\}, & \al\in [0,r],
\\
\emptyset, & \al\in (r,1].
  \end{array}
\right.\]
Then $S_{u_n}=r$ and $[u_n]_r\not=\emptyset$. So $u_n\in F^{r}_{USC}(Y)$.

Define $u\in F_{USC}(Y)$ by putting \[
[u]_\al=\left\{
  \begin{array}{ll}
    (1, 2-\frac{\al}{r}], & \al\in [0,r],
\\
\emptyset, & \al\in (r,1].
  \end{array}
\right.\]
Then $S_u=r$ and $[u]_r = \emptyset$. So
$u\in F^{'r}_{USC}(Y) \setminus F^{r}_{USC}(Y)$.

Clearly ${\rm end}\, u \subseteq {\rm end}\, u_n$.
Observe that
$H_{\rm end}(u_n, u) = H^*( {\rm end}\, u_n, {\rm end}\, u )= \overline{\rho} ((1-\frac{1}{n}, r), \ {\rm end}\, u) =\min\{1/n, r\}$. So
$H_{\rm end}(u_n, u)\to 0$ as $n\to\infty$.
Now we know that the sequence $\{u_n\}$ in $F^{r}_{USC}(Y)$
converges to $u\in F_{USC}(Y)\setminus F^{r}_{USC}(Y)$ according to $H_{\rm end}$ metric.
Thus $F^{r}_{USC}(Y)$ is not a closed subset of $(F_{USC}(Y), H_{\rm end})$.

}
\end{eap}

\begin{lm}\label{cpu}
Let $r\in [0,1]$ and
  let $U$ be a subset of $F^r_{USCG} (X)$. Then the following
   (\romannumeral1-1) is equivalent to (\romannumeral1-2), and
     (\romannumeral2-1) is equivalent to (\romannumeral2-2).
  \\
   (\romannumeral1-1) \ $U$ is relatively compact in $(F_{USCG} (X), H_{\rm end})$.
  \\
  (\romannumeral1-2) \ $U$ is relatively compact in $(F^r_{USCG} (X), H_{\rm end})$.
  \\
    (\romannumeral2-1) \ $U$ is closed in $(F_{USCG} (X), H_{\rm end})$.
\\
     (\romannumeral2-2) \ $U$ is closed in $(F^r_{USCG} (X), H_{\rm end})$.
\end{lm}

\begin{proof}
   Clause (\romannumeral3) of Proposition \ref{ace}
  says that
  $F^{r}_{USCG}(X)$ is a closed subset of $(F_{USCG}(X), H_{\rm end})$.
  From this we obtain that
   (\romannumeral1-1)$\Leftrightarrow$(\romannumeral1-2), and
     (\romannumeral2-1)$\Leftrightarrow$(\romannumeral2-2).

\end{proof}

In this paper, we suppose that $(r,r]=\emptyset$ for $r\in \mathbb{R}$.

\begin{lm}\label{pmu}
Let $r\in [0,1]$ and
  let $U$ be a subset of $F^r_{USCG} (X)$. Then the following
   (\romannumeral1-1) is equivalent to (\romannumeral1-2), (\romannumeral2-1) is equivalent to (\romannumeral2-2), and
     (\romannumeral3-1) is equivalent to (\romannumeral3-2).
  \\
   (\romannumeral1-1) \
$U(\al)$
is relatively compact in $(X, d)$ for each $\al \in (0,1]$.
\\
  (\romannumeral1-2) \
$U(\al)$
is relatively compact in $(X, d)$ for each $\al \in (0,r]$.
  \\
   (\romannumeral2-1)
$U(\al)$
is totally bounded in $(X,d)$ for each $\al \in (0,1]$.
\\
  (\romannumeral2-2) $U(\al)$
is totally bounded in $(X,d)$ for each $\al \in (0,r]$.
\\
 (\romannumeral3-1) \
$U(\al)$
is compact in $(X, d)$ for each $\al \in (0,1]$.
\\
  (\romannumeral3-2) \
$U(\al)$
is compact in $(X, d)$ for each $\al \in (0,r]$.
\end{lm}

\begin{proof}
  Observe that if $\al\in (r,1]$ then
$U(\al) = \emptyset$.
  From this we obtain that
   (\romannumeral1-1)$\Leftrightarrow$(\romannumeral1-2),
    (\romannumeral2-1)$\Leftrightarrow$(\romannumeral2-2), and
     (\romannumeral3-1)$\Leftrightarrow$(\romannumeral3-2).

\end{proof}

\begin{tl}
 \label{crcfegnum}
 Let $r\in [0,1]$ and
  let $U$ be a subset of $F^r_{USCG} (X)$. Then the
  following properties are equivalent.
  \\
   (\romannumeral1) \ $U$ is relatively compact in $(F_{USCG} (X), H_{\rm end})$.
  \\
  (\romannumeral2) \ $U$ is relatively compact in $(F^r_{USCG} (X), H_{\rm end})$.
\\
  (\romannumeral3) \
$U(\al)$
is relatively compact in $(X, d)$ for each $\al \in (0,1]$.
\\
  (\romannumeral4) \
$U(\al)$
is relatively compact in $(X, d)$ for each $\al \in (0,r]$.
\end{tl}

\begin{proof}
  By Lemma \ref{cpu}, (\romannumeral1)$\Leftrightarrow$(\romannumeral2).
 From this and Theorem \ref{rcfegnum}, we obtain that
  (\romannumeral2)$\Leftrightarrow$(\romannumeral3).
  By Lemma \ref{pmu},
 (\romannumeral3)$\Leftrightarrow$(\romannumeral4), and the proof
  is complete.

\end{proof}

\begin{tl} \label{ctbfegnu}
   Let $r\in [0,1]$ and
   let $U$ be a subset of $F^r_{USCG} (X)$. Then
   the following properties are equivalent.
     \\
    (\romannumeral1) $U$ is totally bounded in $(F_{USCG} (X), H_{\rm end})$.
   \\
    (\romannumeral2) $U$ is totally bounded in $(F^r_{USCG} (X), H_{\rm end})$.
\\
  (\romannumeral3)
$U(\al)$
is totally bounded in $(X,d)$ for each $\al \in (0,1]$.
\\
  (\romannumeral4) $U(\al)$
is totally bounded in $(X,d)$ for each $\al \in (0,r]$.
\end{tl}

\begin{proof}
 Clearly (\romannumeral1)$\Leftrightarrow$(\romannumeral2).
From this and Theorem \ref{tbfegnu}, we obtain that
  (\romannumeral2)$\Leftrightarrow$(\romannumeral3).
  By Lemma \ref{pmu}, (\romannumeral3)$\Leftrightarrow$(\romannumeral4),
  and the proof is complete.

\end{proof}

\begin{tl}\label{ncfegum}
Let $r\in [0,1]$ and
   let $U$ be a subset of $F^r_{USCG} (X)$. Then the following properties are equivalent.
   \\
(\romannumeral1) $U$ is compact in $(F_{USCG} (X), H_{\rm end})$
\\
(\romannumeral2) $U$ is compact in $(F^r_{USCG} (X), H_{\rm end})$.
\\
 (\romannumeral3) $U(\al)$
is relatively compact in $(X, d)$ for each $\al \in (0,1]$ and $U$ is closed in $(F_{USCG} (X), H_{\rm end})$;
\\
(\romannumeral4)  $U(\al)$
is compact in $(X, d)$ for each $\al \in (0,1]$ and $U$ is closed in $(F_{USCG} (X), H_{\rm end})$.
\\
 (\romannumeral5) $U(\al)$
is relatively compact in $(X, d)$ for each $\al \in (0,r]$ and $U$ is closed in $(F^r_{USCG} (X), H_{\rm end})$.
\\
(\romannumeral6)  $U(\al)$
is compact in $(X, d)$ for each $\al \in (0,r]$ and $U$ is closed in $(F^r_{USCG} (X), H_{\rm end})$.
\end{tl}

\begin{proof}

   Clearly (\romannumeral1)$\Leftrightarrow$(\romannumeral2). From this and Theorem \ref{cfegum}, we obtain that
(\romannumeral2)$\Leftrightarrow$(\romannumeral3)$\Leftrightarrow$(\romannumeral4).

By Lemma \ref{cpu},
 $U$ is closed in $(F_{USCG} (X), H_{\rm end})$
  if and only if
   $U$ is closed in $(F^r_{USCG} (X), H_{\rm end})$.
   By Lemma \ref{pmu},
   $U(\al)$
is relatively compact in $(X, d)$ for each $\al \in (0,1]$
  if and only if $U(\al)$
is relatively compact in $(X, d)$ for each $\al \in (0,r]$.
   So (\romannumeral3)$\Leftrightarrow$(\romannumeral5).

Similarly, from Lemmas \ref{cpu} and \ref{pmu},
we have that
(\romannumeral4)$\Leftrightarrow$(\romannumeral6).

So
(\romannumeral1)$\Leftrightarrow$(\romannumeral2)$\Leftrightarrow$(\romannumeral3)$\Leftrightarrow$(\romannumeral4)$\Leftrightarrow$(\romannumeral5)$\Leftrightarrow$(\romannumeral6).

\end{proof}

\begin{re}
{\rm
  From Corollary \ref{ncfegum}, Lemmas \ref{cpu} and \ref{pmu},
  the following properties are equivalent.
     \\
(\romannumeral1) $U$ is compact in $(F_{USCG} (X), H_{\rm end})$.
\\
(\romannumeral2) $U$ is compact in $(F^r_{USCG} (X), H_{\rm end})$.
  \\
(\romannumeral3) At least one of
  (\romannumeral1-1),(\romannumeral1-2),
     (\romannumeral3-1) and (\romannumeral3-2) in Lemma \ref{pmu} holds,
 and
at least one of
     (\romannumeral2-1) and (\romannumeral2-2) in Lemma \ref{cpu} holds.
 \\
(\romannumeral4) All of
  (\romannumeral1-1),(\romannumeral1-2),
     (\romannumeral3-1) and (\romannumeral3-2) in Lemma \ref{pmu} hold,
 and
all of
     (\romannumeral2-1) and (\romannumeral2-2) in Lemma \ref{cpu} hold.
  }
\end{re}

\section{An application on relationship
between $H_{\rm end}$ metric and $\Gamma$-convergence
}

As an application of the characterizations of relative compactness,
total boundedness and compactness
given in Section \ref{cmg}, we discuss the relationship
between $H_{\rm end}$ metric and $\Gamma$-convergence
on fuzzy sets.

\begin{pp}\label{hger}
Let $S$ be a nonempty subset of $F_{USC} (X)$.
Let $u$ be a fuzzy set in $S$, and let
 $\{u_n\}$ be a fuzzy set sequence in $S$.
 Then
the following properties are equivalent.
\\
(\romannumeral1) \ $H_{\rm end} (u_n, u) \to 0$ as $n \to \infty$.
\\
(\romannumeral2) \
$\lim_{n\to\infty}^{(\Gamma)} u_n = u$,
and
 $\{u_n, n=1,2,\ldots\}$ is a relatively compact set in $(F_{USC} (X), H_{\rm end})$.
\\
(\romannumeral3) \
$\lim_{n\to\infty}^{(\Gamma)} u_n = u$,
and
 $\{u_n, n=1,2,\ldots\}$ is a relatively compact set in $(S, H_{\rm end})$.
\\
(\romannumeral4) \
$\lim_{n\to\infty}^{(\Gamma)} u_n = u$,
and
  $\{u_n, n=1,2,\ldots\} \cup \{u\}$ is a compact set in $(S, H_{\rm end})$.
\\
(\romannumeral5) \
$\lim_{n\to\infty}^{(\Gamma)} u_n = u$,
and
  $\{u_n, n=1,2,\ldots\} \cup \{u\}$ is a compact set in $(F_{USC} (X), H_{\rm end})$.
\end{pp}

\begin{proof}

To show (\romannumeral1)$\Rightarrow$(\romannumeral5).
Assume that (\romannumeral1) is true. By Theorem \ref{hkg} and Remark \ref{hmr}, $\lim_{n\to\infty}^{(\Gamma)} u_n = u$.
Clearly $\{u_n, n=1,2,\ldots\} \cup \{u\}$ is a compact set in
$(F_{USC} (X), H_{\rm end})$.
So
(\romannumeral5) is true.

It can be seen that $\{u_n, n=1,2,\ldots\} \cup \{u\}$ is a compact set in
$(F_{USC} (X), H_{\rm end})$
if and only if
$\{u_n, n=1,2,\ldots\} \cup \{u\}$ is a compact set in
$(S, H_{\rm end})$.
So (\romannumeral5)$\Leftrightarrow$(\romannumeral4).

If $\{u_n, n=1,2,\ldots\} \cup \{u\}$ is a compact set in $(S, H_{\rm end})$,
then
$\{u_n, n=1,2,\ldots\}$ is relatively compact in $(S, H_{\rm end})$
because
$\{u_n, n=1,2,\ldots\} $ is a subset of $\{u_n, n=1,2,\ldots\} \cup \{u\}$.
So
 (\romannumeral4)$\Rightarrow$(\romannumeral3).

Clearly if $\{u_n, n=1,2,\ldots\}$ is a relatively compact set in $(S, H_{\rm end})$,
then $\{u_n, n=1,2,\ldots\}$ is a relatively compact set in $(F_{USC}(X), H_{\rm end})$.
So
 (\romannumeral3)$\Rightarrow$(\romannumeral2).

To show (\romannumeral2)$\Rightarrow$(\romannumeral1), we proceed by contradiction.
Assume
that (\romannumeral2) is true. If (\romannumeral1) is not true;
that is, $H_{\rm end} (u_n, u) \not\to 0$.
Then there is an $\varepsilon>0$ and a subsequence
$\{v_n^{(1)}\}$ of $\{u_n\}$
that
\begin{equation}\label{scre}
  H_{\rm end} (v_n^{(1)}, u) \geq \varepsilon \mbox{ for all } n=1,2,\ldots.
\end{equation}
Since
 $\{u_n, n=1,2,\ldots\}$ is relatively compact in $(F_{USC} (X), H_{\rm end})$,
there is a subsequence $\{v_n^{(2)}\}$ of
$\{v_n^{(1)}\}$ and $v\in F_{USC} (X)$
such that $H_{\rm end} (v_n^{(2)}, v) \to 0$.
Hence by Theorem \ref{hkg} and Remark \ref{hmr},
$\lim_{n\to\infty}^{(\Gamma)} v_n^{(2)} = v$.
Since $\lim_{n\to\infty}^{(\Gamma)} u_n = u$, then by Remark \ref{gby},
$u=v$.
So $H_{\rm end} (v_n^{(2)}, u) \to 0$,
which contradicts \eqref{scre}.

Since we have shown (\romannumeral1)$\Rightarrow$(\romannumeral5),
(\romannumeral5)$\Leftrightarrow$(\romannumeral4),
(\romannumeral4)$\Rightarrow$(\romannumeral3),
(\romannumeral3)$\Rightarrow$(\romannumeral2)
and (\romannumeral2)$\Rightarrow$(\romannumeral1),
the proof is complete.

We can also show this theorem as follows.
First we show that
 (\romannumeral1)$\Leftrightarrow$(\romannumeral3)
$\Leftrightarrow$(\romannumeral4) by verifying that
 (\romannumeral1)$\Rightarrow$(\romannumeral4)
$\Rightarrow$(\romannumeral3)$\Rightarrow$(\romannumeral1)
(The proof of (\romannumeral1)$\Rightarrow$(\romannumeral4)
is similar to that of
(\romannumeral1)$\Rightarrow$(\romannumeral5).
The proof of (\romannumeral3)$\Rightarrow$(\romannumeral1)
is similar to that of
(\romannumeral2)$\Rightarrow$(\romannumeral1)).
Then put $S=F_{USC}(X)$, we obtain that
(\romannumeral1)$\Leftrightarrow$(\romannumeral2)
from
(\romannumeral1)$\Leftrightarrow$(\romannumeral3),
and that
(\romannumeral1)$\Leftrightarrow$(\romannumeral5)
from
(\romannumeral1)$\Leftrightarrow$(\romannumeral4).
So we have that
(\romannumeral1), (\romannumeral2),
(\romannumeral3), (\romannumeral4) and (\romannumeral5)
are equivalent to each other.

\end{proof}

\begin{pp}\label{hge}
Let $u$ be a fuzzy set in $F_{USCG} (X)$, and let
 $\{u_n\}$ be a fuzzy set sequence in $F_{USCG} (X)$.
 Then
the following properties are equivalent.
\\
(\romannumeral1) \ $H_{\rm end} (u_n, u) \to 0$ as $n \to \infty$.
\\
(\romannumeral2) \
$\lim_{n\to\infty}^{(\Gamma)} u_n = u$,
and
 for each $\al \in (0,1]$, $\bigcup_{n=1}^{+\infty} [u_n]_\al$ is relatively compact in $(X,d)$.
\\
(\romannumeral3) \
$\lim_{n\to\infty}^{(\Gamma)} u_n = u$,
and
 for each $\al \in (0,1]$,
$\bigcup_{n=1}^{+\infty} [u_n]_\al \cup [u]_\al $ is compact in $(X,d)$.
\\
(\romannumeral4) \
$\lim_{n\to\infty}^{(\Gamma)} u_n = u$,
$\{u_n, n=1,2,\ldots\} \cup \{u\}$ is closed in $(F_{USCG} (X), H_{\rm end})$,
and
 for each $\al \in (0,1]$,
$\bigcup_{n=1}^{+\infty} [u_n]_\al \cup [u]_\al $ is compact in $(X,d)$.
\end{pp}

\begin{proof}
The desired result follows from Proposition \ref{hger}, Theorem \ref{rcfegnum} and Theorem \ref{cfegum}.
The proof is routine.

Put $S=F_{USCG}(X)$ in Proposition \ref{hger}.
Then we obtain that the following conditions (a), (b) and (c) are equivalent.
\\
(a) $H_{\rm end} (u_n, u) \to 0$.
\\
(b)
$\lim_{n\to\infty}^{(\Gamma)} u_n = u$,
and
$\{u_n, n=1,2,\ldots\}$ is a relatively compact set in $(F_{USCG} (X), H_{\rm end})$.
\\
(c) $\lim_{n\to\infty}^{(\Gamma)} u_n = u$,
and
   $\{u_n, n=1,2,\ldots\} \cup \{u\}$ is a compact set in $(F_{USCG} (X), H_{\rm end})$.

(a) is (\romannumeral1). By Theorem \ref{rcfegnum}, (b)$\Leftrightarrow$(\romannumeral2).
By Theorem \ref{cfegum}, (c)$\Leftrightarrow$(\romannumeral4).
We can see that
(\romannumeral4)$\Rightarrow$(\romannumeral3)
$\Rightarrow$(\romannumeral2).
So from
(a)$\Leftrightarrow$(b)$\Leftrightarrow$(c),
we have that
(\romannumeral1)$\Leftrightarrow$(\romannumeral2)$\Leftrightarrow$
(\romannumeral3)$\Leftrightarrow$(\romannumeral4).

\end{proof}

\begin{pp} \label{csm}
The following statements are equivalent:
\\
 (\rmn1) $X$ is compact.
\\
(\rmn2) Let $\{u_n\}$ be a sequence in $F_{USC}(X)$ and $u\in F_{USC}(X)$.
Then $H_{\rm end} (u_n, u) \to 0$ as $n \to \infty$ if and only if $\lim_{n\to\infty}^{(\Gamma)} u_n = u$.
\\
(\rmn3) Let $\{u_n\}$ be a sequence in $F_{USCG}(X)$ and $u\in F_{USCG}(X)$.
Then $H_{\rm end} (u_n, u) \to 0$ as $n \to \infty$ if and only if $\lim_{n\to\infty}^{(\Gamma)} u_n = u$.
\\
(\rmn4)  Let $\{u_n\}$ be a sequence in $F_{USCB}(X)$ and $u\in F_{USCB}(X)$.
Then $H_{\rm end} (u_n, u) \to 0$ as $n \to \infty$ if and only if $\lim_{n\to\infty}^{(\Gamma)} u_n = u$.
 \end{pp}

\begin{proof}
 Assume that (\rmn1) is true. Suppose that $\{u_n\}$ is a sequence in $F_{USC}(X)$ and $u\in F_{USC}(X)$. Then $u\in F_{USCB}(X)$ and $\{u_n\}\subseteq F_{USCB}(X)$. Clearly for each $\al \in (0,1]$, $\bigcup_{n=1}^{+\infty} [u_n]_\al$ is relatively compact in $(X,d)$.
Thus by Proposition \ref{hge}, $H_{\rm end} (u_n, u) \to 0$ as $n \to \infty$ if and only if $\lim_{n\to\infty}^{(\Gamma)} u_n = u$.
So (\rmn2) is true. Thus (\rmn1)$\Rightarrow$(\rmn2).
Clearly (\rmn2)$\Rightarrow$(\rmn3)$\Rightarrow$(\rmn4).
To complete the proof, we only need to show
(\rmn4)$\Rightarrow$(\rmn1).

Assume that (\rmn1) is not true.
Then there is a sequence $\{x_n\}$ in $X$ which has no convergent subsequence in $X$.
For $n=1,2,\ldots$, let $S_n=\{x_1\}\cup \{x_n\}$ and
$u_n=\chi_{S_n}$.
Then $\{u_n\} \subseteq F_{USCB}(X)$ and $\lim_{n\to\infty}^{(\Gamma)} u_n = u_1$.

For each $n\in \mathbb{N}$, by \eqref{efn},
$H_{\rm end} (u_n, u_1)=\min\{1, H(S_n, S_1)\} =\min\{1, d(x_n, x_1)\}$.
Obviously $d(x_n, x_1)\not\to 0$.
This means that
$H_{\rm end} (u_n, u_1)\not\to 0$. So (\rmn4) is not true.
Thus (\rmn4)$\Rightarrow$(\rmn1).
This completes the proof.

\end{proof}

\begin{pp} \label{csme}
The following statements are equivalent:
\\
 (\rmn1) $X$ is compact.
\\
(\rmn2) Let $\{C_n\}$ be a sequence in $C(X)\cup\{\emptyset\}$ and $C\in C(X)\cup\{\emptyset\}$.
Then $H (C_n, C) \to 0$ as $n \to \infty$ if and only if $\lim_{n\to\infty}^{(K)} C_n = C$.
\\
(\rmn3)  Let $\{C_n\}$ be a sequence in $K(X)$ and $C\in K(X)$.
Then $H  (C_n, C) \to 0$ as $n \to \infty$ if and only if $\lim_{n\to\infty}^{(K)} C_n = C$.
 \end{pp}

\begin{proof}
  Note that for a set $A\in X$, $A\in C(X)\cup \emptyset$
if and only if $\chi_A\in F_{USC}(X)$.
Thus by Proposition\ref{bce}(\rmn1)(\rmn2) and Proposition \ref{csm},
(\rmn1)$\Rightarrow$(\rmn2).
Clearly (\rmn2)$\Rightarrow$(\rmn3).
To complete the proof, we only need to show that (\rmn3)$\Rightarrow$(\rmn1).

Assume that (\rmn1) is not true.
Then there is a sequence $\{x_n\}$ in $X$ which has no convergent subsequence in $X$.
For $n=1,2,\ldots$, let $S_n=\{x_1\}\cup \{x_n\}$.
Then $\{S_n\} \subseteq K(X)$ and $\lim_{n\to\infty}^{(K)} S_n = S_1$.
For each $n\in \mathbb{N}$,
$H(S_n, S_1) = d(x_n, x_1)$.
Obviously $d(x_n, x_1)\not\to 0$. This means
 $H(S_n, S_1)\not\to 0$. So (\rmn3) is not true.
Thus (\rmn3)$\Rightarrow$(\rmn1).
This completes the proof.

\end{proof}

\begin{pp}\label{hgermy}
Let $\mathcal{D}$ be a nonempty subset of $C(X)$.
Let $A$ be a set in $\mathcal{D}$, and let
 $\{A_n\}$ be a sequence of sets in $\mathcal{D}$.
 Then
the following properties are equivalent.
\\
(\romannumeral1) \ $H (A_n, A) \to 0$ as $n \to \infty$.
\\
(\romannumeral2) \
$\lim_{n\to\infty}^{(K)} A_n = A$,
and
 $\{A_n, n=1,2,\ldots\}$ is a relatively compact set in $(C(X), H)$.
\\
(\romannumeral3) \
$\lim_{n\to\infty}^{(K)} A_n = A$,
and
 $\{A_n, n=1,2,\ldots\}$ is a relatively compact set in $(\mathcal{D}, H)$.
\\
(\romannumeral4) \
$\lim_{n\to\infty}^{(K)} A_n = A$,
and
  $\{A_n, n=1,2,\ldots\} \cup \{A\}$ is a compact set in $(\mathcal{D}, H)$.
\\
(\romannumeral5) \
$\lim_{n\to\infty}^{(K)} A_n = A$,
and
  $\{A_n, n=1,2,\ldots\} \cup \{A\}$ is a compact set in $(C(X), H)$.
\end{pp}

\begin{proof}
  The proof is similar to that of Proposition \ref{hger}.

\end{proof}

\begin{pp}\label{hgermyc}
Let $A$ be a set in $K(X)$, and let
 $\{A_n\}$ be a sequence of sets in $K(X)$.
 Then
the following properties are equivalent.
\\
(\romannumeral1) \ $H (A_n, A) \to 0$ as $n \to \infty$.
\\
(\romannumeral2) \
$\lim_{n\to\infty}^{(K)} A_n = A$,
and
 $\bigcup_{n=1}^{+\infty} A_n$ is a relatively compact set in $(X, d)$.
\\
(\romannumeral3) \
$\lim_{n\to\infty}^{(K)} A_n = A$,
and
  $\bigcup_{n=1}^{+\infty} A_n \cup A$ is a compact set in $(X, d)$.
  \\
(\romannumeral4) \
$\lim_{n\to\infty}^{(K)} A_n = A$,
  $\bigcup_{n=1}^{+\infty} A_n \cup A$ is a compact set in $(X, d)$,
    and
  $\{A_n, n=1,2,\ldots\} \cup \{A\}$ is a closed set in $(K(X), H)$.
\end{pp}

\begin{proof}
  The desired result follows from Proposition \ref{hgermy} and Theorems \ref{rce} and \ref{come}. The proof is routine and similar to that
  of Proposition \ref{hge}.

Put $\mathcal{D}=K(X)$ in Proposition \ref{hgermy}.
 Then we obtain that
the following conditions (a), (b) and (c)  are equivalent.
\\
(a) \ $H (A_n, A) \to 0$ as $n \to \infty$.
\\
(b) \
$\lim_{n\to\infty}^{(K)} A_n = A$,
and
 $\{A_n, n=1,2,\ldots\}$ is a relatively compact set in $(K(X), H)$.
\\
(c) \
$\lim_{n\to\infty}^{(K)} A_n = A$,
and
  $\{A_n, n=1,2,\ldots\} \cup \{A\}$ is a compact set in $(K(X), H)$.

(a) is (\romannumeral1). By Theorem \ref{rce}, (b)$\Leftrightarrow$(\romannumeral2).
By Theorem \ref{come}, (c)$\Leftrightarrow$(\romannumeral4).
We can see that
(\romannumeral4)$\Rightarrow$(\romannumeral3)
$\Rightarrow$(\romannumeral2).
So from
(a)$\Leftrightarrow$(b)$\Leftrightarrow$(c),
we have that
(\romannumeral1)$\Leftrightarrow$(\romannumeral2)$\Leftrightarrow$
(\romannumeral3)$\Leftrightarrow$(\romannumeral4).

\end{proof}

 \begin{re}
   {\rm
  Let $u\in F_{USCG} (X)$ and $\{u_n\}$ be a fuzzy set sequence in $F_{USC} (X)$.
 Let $\al\in (0,1]$.
Since $[u]_\al$ is compact in $X$,
we have that the conditions (a) $\bigcup_{n=1}^{+\infty} [u_n]_\al$ is relatively compact in $(X,d)$,
   and (b) $\bigcup_{n=1}^{+\infty} [u_n]_\al \cup [u]_\al$ is relatively compact in $(X,d)$,
are equivalent.

So
``$\bigcup_{n=1}^{+\infty} [u_n]_\al$ is relatively compact in $(X,d)$'' can be
   replaced by
   ``$\bigcup_{n=1}^{+\infty} [u_n]_\al \cup [u]_\al$ is relatively compact in $(X,d)$''
   in clause (\romannumeral2) of Proposition \ref{hge}.

Similar replacement can be made in Propositions \ref{hger},
 \ref{hgermy} and \ref{hgermyc}.

   }
 \end{re}

Propositions \ref{hger}, \ref{hge}, \ref{hgermy} and \ref{hgermyc}
can be shown in different ways.
Below we give some other proofs.

Proposition \ref{hgermy} implies Proposition \ref{hger}.

Let $S$ be a nonempty subset of $F_{USC} (X)$.
Let $u$ be a fuzzy set in $S$, and let
 $\{u_n\}$ be a fuzzy set sequence in $S$.

Put $A={\rm end} \, u$, and for $n=1,2,\ldots$, put $A_n = {\rm end} \, u_n$
in
Proposition \ref{hgermy}.

Put $\mathcal{D} = \{{\rm end}\, u: u\in F_{USC}(X)\}$
in
Proposition \ref{hgermy}.
Then from (\romannumeral1)$\Leftrightarrow$(\romannumeral3) in Proposition \ref{hgermy}, we obtain that (\romannumeral1)$\Leftrightarrow$(\romannumeral2) in Proposition \ref{hger}.

Put $\mathcal{D} = \{{\rm end}\, u: u\in S\}$
in
Proposition \ref{hgermy}.
Then from (\romannumeral1)$\Leftrightarrow$(\romannumeral3) in Proposition \ref{hgermy}, we obtain that (\romannumeral1)$\Leftrightarrow$(\romannumeral3) in Proposition \ref{hger}.

Similarly, we can show that (\romannumeral1)$\Leftrightarrow$(\romannumeral4)
and (\romannumeral1)$\Leftrightarrow$(\romannumeral5) in Proposition \ref{hger}.
So Proposition \ref{hgermy} implies Proposition \ref{hger}.

\vspace{1.6mm}
Proposition \ref{hgermy}, Theorem \ref{rcfegnum} and Theorem \ref{cfegum} imply Proposition \ref{hge}.

Let $u$ be a fuzzy set in $F_{USCG} (X)$, and let
 $\{u_n\}$ be a fuzzy set sequence in $F_{USCG} (X)$.

Put $A={\rm end} \, u$, and for $n=1,2,\ldots$, put $A_n = {\rm end} \, u_n$
in
Proposition \ref{hgermy}.

Put $\mathcal{D} = \{{\rm end}\, u: u\in F_{USCG}(X)\}$
in
Proposition \ref{hgermy}.
Then from (\romannumeral1)$\Leftrightarrow$(\romannumeral3) in Proposition \ref{hgermy} and Theorem \ref{rcfegnum}, we obtain that (\romannumeral1)$\Leftrightarrow$(\romannumeral2) in Proposition \ref{hge}.

Similarly from (\romannumeral1)$\Leftrightarrow$(\romannumeral4) in Proposition \ref{hgermy} and Theorem \ref{cfegum},
we obtain that (\romannumeral1)$\Leftrightarrow$(\romannumeral4) in Proposition \ref{hge}.
Since
(\romannumeral4)$\Rightarrow$(\romannumeral3)
$\Rightarrow$(\romannumeral2) in Proposition \ref{hge},
it follows that
(\romannumeral1)$\Leftrightarrow$(\romannumeral2)$\Leftrightarrow$(\romannumeral3)$\Leftrightarrow$(\romannumeral4)
 in Proposition \ref{hge}.
So Proposition \ref{hgermy}, Theorem \ref{rcfegnum} and Theorem \ref{cfegum} imply Proposition \ref{hge}.

\vspace{1.6mm}
 Proposition \ref{hger} implies Propositions \ref{hgermy} and \ref{hgermyc}.

 Proposition \ref{hge}
implies
 Proposition \ref{hgermyc}.

\begin{pp}
\label{bce}
(\romannumeral1) Let $A\in C(X)$ and
 $\{A_n\}$ a sequence in $C(X)$.
Then $H_{\rm end}(\chi_{A_n}, \chi_A) \to 0$ as $n \to \infty$
if and only if
$H (A_n, A) \to 0$ as $n \to \infty$.
  \\
(\romannumeral2)
 Let $A\in P(X)$ and
 $\{A_n\}$ a sequence in $P(X)$.
Then
 $\lim_{n\to\infty}^{(\Gamma)} \chi_{A_n} = \chi_{A}$
if and only if $\lim_{n\to\infty}^{(K)} A_n = A$.
\\
(\romannumeral3)
Let $\mathcal{D}$ be a subset of $C(X)$ and
 $\mathcal{B}$ a subset of $\mathcal{D}$. Then
$\mathcal{B}$ is totally bounded (respectively, relatively compact, compact, closed) in $(\mathcal{D}, H)$
if and only if
$\mathcal{B}_{F(X)}$ is totally bounded(respectively, relatively compact, compact, closed) in $(\mathcal{D}_{F(X)}, H_{\rm end})$.

\end{pp}

\begin{proof}
  (\romannumeral1) and (\romannumeral3) follow immediately from \eqref{efn}.
  (\romannumeral2) follows
  from the definition of Kuratowski convergence and $\Gamma$-convergence.

\end{proof}

Let $\mathcal{D}$ be a nonempty subset of $C(X)$.
Let $A$ be a set in $C(X)$, and let
 $\{A_n\}$ be a sequence of sets in $\mathcal{D}$.

Let $S=C(X)_{F(X)}$ in Proposition \ref{hger}.
Then
from (\romannumeral1)$\Leftrightarrow$(\romannumeral3) in
Proposition \ref{hger},
we have that:
\\
(c-1) $H_{\rm end} (\chi_{A_n}, \chi_{A}) \to 0$ as $n \to \infty$
if and only if
$\lim_{n\to\infty}^{(\Gamma)} \chi_{A_n} = \chi_{A}$,
and
 $\{\chi_{A_n}, n=1,2,\ldots\}$ is a relatively compact set in $(C(X)_{F(X)}, H_{\rm end})$.

 By Proposition \ref{bce},
(c-1) means that (\romannumeral1)$\Leftrightarrow$(\romannumeral2) in Proposition \ref{hgermy}.

Let $S=\mathcal{D}_{F(X)}$ in Proposition \ref{hger}.
Then
from (\romannumeral1)$\Leftrightarrow$(\romannumeral3) in
Proposition \ref{hger},
we have that:
\\
(c-2) $H_{\rm end} (\chi_{A_n}, \chi_{A}) \to 0$ as $n \to \infty$
if and only if
$\lim_{n\to\infty}^{(\Gamma)} \chi_{A_n} = \chi_{A}$,
and
 $\{\chi_{A_n}, n=1,2,\ldots\}$ is a relatively compact set in $(\mathcal{D}_{F(X)}, H_{\rm end})$.

  By Proposition \ref{bce},
(c-2) means that (\romannumeral1)$\Leftrightarrow$(\romannumeral3) in Proposition \ref{hgermy}.

Similarly, by using Proposition \ref{bce}, we can show that (\romannumeral1)$\Leftrightarrow$(\romannumeral4) in Proposition \ref{hger} implies that
(\romannumeral1)$\Leftrightarrow$(\romannumeral4)
and
(\romannumeral1)$\Leftrightarrow$(\romannumeral5) in Proposition \ref{hgermy}.

By clauses (\romannumeral1) and (\romannumeral2) of Proposition \ref{bce}
and clause (\romannumeral3) of Proposition \ref{crpn},
we can see that Proposition \ref{hge}
implies
 Proposition \ref{hgermyc}.

By using level characterizations of $H_{\rm end}$ and $\Gamma$-convergence on fuzzy sets, it is easy to show that Proposition \ref{hgermyc}
also implies
Proposition \ref{hge}.

\section{Conclusion}

In this paper, we present the characterizations of total boundedness, relative compactness and compactness in $(F_{USCG} (X), H_{\rm end})$.
Here $X$ is a general metric space.
Based on this, we also give the characterizations of total boundedness, relative compactness and compactness in $(F^r_{USCG} (X), H_{\rm end})$, $r\in [0,1]$.
 $(F^r_{USCG} (X), H_{\rm end})$, $r\in [0,1]$, are metric subspaces
 of $(F_{USCG} (X), H_{\rm end})$.

The conclusions in this paper
 significantly improve
the corresponding conclusions given in our previous paper
\cite{huang}.
Therein we give the characterizations of total boundedness, relative compactness and compactness in $(F_{USCG} (\mathbb{R}^m), H_{\rm end})$.
$\mathbb{R}^m$ is a special type of metric space.

We discuss the relationship
between $H_{\rm end}$ metric and $\Gamma$-convergence
as an application of the characterizations of relative compactness,
total boundedness and compactness
given in this paper.

The results in this paper have potential applications in the research
of
fuzzy sets involved the endograph metric and the $\Gamma$-convergence.

\section*{Acknowledgement}

The author would like to thank the two anonymous referees
for their invaluable comments and suggestions
which improves the presentation of this paper.

\section*{Compliance with Ethical Standards}

Funding: This study was funded by Natural Science Foundation of Fujian Province of China (No. 2020J01706).

Conflict of Interest: The author declares that she has no conflict of interest.

Ethical approval: This article does not contain any studies with human participants or animals performed by the author.

\end{document}